\documentclass[11pt]{article}
\usepackage{amsmath,amssymb,amscd,amsthm,graphicx,enumerate,enumitem,framed}
\usepackage[a4paper]{geometry}
\usepackage{hyperref}

\title{Mean Curvature Flow Of Reifenberg  Sets}
\author{Or Hershkovits\thanks{The author was partially supported by NSF grants DMS 1406407 and DMS 1105656}}
\date{\today}

\usepackage{amsfonts}

\numberwithin{equation}{section}
  \theoremstyle{plain}
 \newtheorem{theorem}[equation]{Theorem}

\newtheorem{corollary}[equation]{Corollary}
 \newtheorem{lemma}[equation]{Lemma}
 
 \theoremstyle{remark}
 \newtheorem{remark}[equation]{Remark}
 \theoremstyle{remark}
\theoremstyle{definition}
 \newtheorem{definition}[equation]{Definition}

\newcommand*{\rom}[1]{\expandafter\@slowromancap\romannumeral #1@}

\newcommand{\de}{\varepsilon}

\newcommand{\osc}{\mathrm{osc}}
\newcommand{\spa}{\mathrm{span}}

\begin{document}
\maketitle

\begin{abstract}
In this paper, we prove short time existence and uniqueness of smooth evolution by mean curvature in $\mathbb{R}^{n+1}$ starting from any $n$-dimensional $(\de,R)$-Reifenberg flat set with $\de$ sufficiently small. More precisely, we show that the level set flow in such a situation is non-fattening and smooth. These sets have a weak metric notion of tangent planes at every small scale, but the tangents are allowed to tilt as the scales vary. As this class is wide enough to include some fractal sets, this provides the first example (when $n>1$) of unique smoothing  by mean curvature flow of sets with Hausdorff dimension $> n$.   
\end{abstract}

\tableofcontents

\section{Introduction}
A family of smooth embeddings  $\phi_t:M^{n} \rightarrow \mathbb{R}^{n+1}$ for $t\in (a,b)$ is said to evolve by mean curvature if it satisfies the equation
\begin{equation}
\frac{d}{dt}\phi_t(x)=\vec{H}(\phi_t(x)),
\end{equation}  
where $\vec{H}$ is the mean curvature vector. Equivalently, by the first variation formula, mean curvature flow is the negative gradient flow of the area functional.

 If a compact hypersurface $M \subseteq \mathbb{R}^{n+1}$ is of type $C^2$, it follows from standard parabolic PDE theory that there exists a unique mean curvature flow starting from $M$ for some finite maximal time $T$, and that in fact
\begin{equation}
\lim_{t\rightarrow T}\max_{x\in M_t}|A(x,t)|=\infty 
\end{equation}
 (see for instance \cite{Man}).  

\vspace{5 mm}

The question of mean curvature flow (and geometric flows in general) with rough initial data, i.e when the $C^2$ assumption is weakened, has been researched extensively (see e.g \cite{EH2},\cite{EH},\cite{Wan},\cite{Sim2},\cite{KL},\cite{Lau}). In the case that $M$ is merely Lipschitz, short time existence was proved by Ecker and Huisken in the celebrated paper \cite{EH}. Their proof is based on the fact that in the $C^1$ case, $M$ can be written locally as a graph of a $C^1$ function, and the ellipticity of the graphical mean curvature equation is controlled by an interior gradient estimate. Note that even in the Lipschitz case, the $n$-dimensional Hausdorff measure is still finite, although the gradient of the area functional may not be. In a different direction, in \cite{Lau} Lauer was recently able to show that when $n=1$, for any Jordan curve $\gamma$ in $\mathbb{R}^2$, if $m(\gamma)=0$ (where $m$ is the two dimensional Lebesgue measure) then the level set flow (see Definition \ref{ls_flow}) is non-fattening and smooth. 

\vspace{5 mm}

The current paper deals with the existence and uniqueness of smooth flows in $\mathbb{R}^{n+1}$ starting from a class of sets which is general enough to include some sets of Hausdorff dimension larger than $n$.
    
\begin{definition}[Reifenberg flat sets \cite{Reif}]
A compact, connected set $X \subseteq \mathbb{R}^{n+1}$ is called $(\de,R)$-\textbf{Reifenberg flat} if for every $x\in X$ and  $0<r<R$ there exists a hyperplane $P$ such that 
\begin{equation}
d_H(B(x,r)\cap P,B(x,r)\cap X) \leq \de r.
\end{equation}
Here $d_H$ is the Hausdorff distance.
\end{definition}

The point is that the approximating hyperplanes may tilt as the scales vary. In \cite{Reif}, Reifenberg showed that provided $\de$ is sufficiently small, an $(\de,R)$-Reifenberg flat set is a topological and, in fact, a $C^\alpha$-manifold. As stated above,  the Reifenberg condition is weak enough to allow some fractal sets. For instance, as described in \cite{Tor}, a variant of the Koch snowflake, at which the angles in the construction are $\beta$ instead of $\frac{\pi}{3}$ is $(\de,R)$-Reifenberg with $\de=\sin\beta$. Note that the snowflake is not graphical at any scale. An analogue of this can be done in every dimension.

\vspace{5 mm}

Before diving into more technicalities, we can already state a form of our main theorem.
\begin{theorem}\label{dummy_main_thm}
There exists some $\de_0,c_0>0$ such that if $X$ is $(\de,R)$-Reifenberg flat for $0<\de<\de_0$ then there exists a smooth solution to the mean curvature flow $(X_t)_{t\in (0,c_0R^2)}$ attaining the initial value $X$ in the following sense: 
\begin{equation}
\lim_{t \rightarrow 0}d_{H}(X,X_t)=0.
\end{equation}
Moreover, the flow $(X_t)$ is unique (in a sense that will be explained shortly).
\end{theorem}   
\noindent Thus, the mean curvature flow provides a canonical smoothing of Reifenberg flat sets.

\vspace{5 mm}

To state he uniqueness result more accurately, we need the following definition:
\begin{definition}[level set flow, \cite{Ilm},\cite{Ilm3}]\label{ls_flow}
A family of closed subsets of $\mathbb{R}^{n+1}$ $(X_t)_{t\in[0,b]}$ is said to be a \textbf{weak set flow} starting from $X_0$ if it satisfies the avoidance principle with respect to any smooth mean curvature flow. More precisely, for any smooth mean curvature flow $(\Delta_t)_{t\in [t_0,t_1]}$ with $0 \leq t_0 \leq t_1 \leq b$ such that
\begin{equation}
 \Delta_{t_0}\cap X_{t_0}= \emptyset
\end{equation}
we have
\begin{equation}
\Delta_t\cap X_t = \emptyset
\end{equation}
for every $t\in [t_0,t_1]$. The \textbf{level set flow} is the maximal weak set flow starting from $X$. 
\end{definition}
The level set flow was defined in \cite{ES1} and \cite{CGG} using the language of viscosity solutions for PDEs in order to develop a theory for weak solutions of mean curvature flow. The above more geometric definition is from \cite{Ilm3}, where the equivalence was also shown to hold. If $X_0$ is a smooth submanifold, the level set flow will coincide with the classical evolution by mean curvature flow for as long as the latter is defined. An advantage of working with the level set flow is that it is defined and unique for all time (so it is indifferent to singularities), and it allows one to flow any closed set. The drawback of it is that the $X_t$ may develop an interior (in $\mathbb{R}^{n+1}$), even if $X_0$ was the boundary of an open set. The development of an interior is referred to as ``fattening'' and is the right notion of non-uniqueness in this setting.

\vspace{5 mm}

We are now ready to state the full version of our main theorem.
\begin{theorem}[Main theorem]\label{real_main_thm}
There exists some $\de_0,c_0>0$ such that if $X$ is $(\de,R)$-Reifenberg flat for $0<\de<\de_0$ then the level set flow starting from $X$, $(X_t)_{t\in [0,c_0R^2]}$  is a (non vanishing) smooth evolution by mean curvature flow for $t\in(0,c_0R^2)$ that satisfies
\begin{equation}
\lim_{t \rightarrow 0}d_{H}(X,X_t)=0.
\end{equation}
In particular, the level set flow does not fatten.
\end{theorem}   

\vspace{10 mm}

We will now give an outline of the argument. Our first goal will be to construct a smooth solution to the mean curvature flow $(X_t)_{t\in(0,c_0R^2)}$ which converges to $X$ in the Hausdorff sense as $t\rightarrow 0$. To do that,  we first approximate the set $X$  by smooth hypersurfaces at each scale, according to the following theorem, implicit in \cite{HW} (see also Section \ref{prelim}).
\begin{theorem}[\cite{HW},\cite{Reif}]\label{aux_reif_thm}
There exist some constants $c_1,c_2>0$ such that if $X$ is $(\de,R)$-Reifenberg flat for $0<\de<\de_0$ then there exists a family of hypersurfaces $(X^r)_{0<r<R/4}$ such that:
\begin{enumerate}
\item  $d_H(X^r,X) \leq c_1\de r$.
\item $|A| \leq \frac{c_2\de}{r}$ for every $x \in X^r$. Here $A$ is the second fundamental form of $X^r$.
\item for every $x\in X$, $r\in(0,R/4)$ and $s\in(r,R/4)$, $B(x,s)\cap X^r$ can be decomposed as
\begin{equation}
B(x,s)\cap X^r=G\cup B
\end{equation}
where $G$ is connected and $B \cap B(x,(1-20\de)s)=\emptyset$.
\end{enumerate}
\end{theorem}

We want to construct a smooth evolution of $X$ by taking a limit of the flows emanating from the $X^r$. In order to do that, we derive the following uniform estimates for the evolutions of the hypersurfaces $X^r$.

\begin{theorem}[Uniform estimates]\label{main_thm2}
 For every $\Lambda>0$ there exist some $\de$ and $c_0,c_1,c_2,c_3$ such that if $X$ is $(\de,R)$-Reifenberg flat, and considering the approximating surfaces $X^r$ from Theorem \ref{aux_reif_thm}, each $X^r$ flows smoothly by mean curvature for time  $t\in [0,c_0 R^2]$ and for every $t\in [c_3r^2,c_0R^2]$ we have:
\begin{enumerate}
\item Denoting by $A^r(t)$  the second fundamental form of $X^r_t$ 
\begin{equation}\label{uni_est_curv}
|A^r(t)| \leq \frac{c_1}{\sqrt{t}}.
\end{equation}
\item $X^r_t$ approximates $X$ in the Hausdorff sense: 
\begin{equation}\label{uni_est_dist}
d_H(X^{r}_t,X) \leq c_2\sqrt{t}.
\end{equation}
\item For every $x\in X$ and  $s\in (\frac{\sqrt{t}}{c_1},R/4)$ we have
\begin{equation}\label{uni_est_conn}
B(x,s)\cap X^r_t=G\cup B
\end{equation}
where $G$ is connected and $B\cap B(x,\frac{9}{10}s)=\emptyset$.
\end{enumerate} 
Moreover, the constants $c_1,c_2$ satisfy
\begin{equation}\label{struct_in1}
\begin{aligned}
&1.\; c_1c_2<\min\{10^{-6},\Lambda^{-1}\}. \\
&2.\; c_1^2<\frac{1}{80}. 
\end{aligned}
\end{equation}
\end{theorem}

For reasons that will become apparent soon, it is convenient to make the following definition.
\begin{definition}[Approximate solutions, Physical solutions]\label{app_phys_sol}
A MCF $(Y_t)_{t\in [0,c_3R^2)}$ that satisfies estimates \eqref{uni_est_curv},\eqref{uni_est_dist} and \eqref{uni_est_conn}  for every $t\in [c_3r^2,c_0R^2]$ will be called a $(\Lambda$) \textbf{ $r$-approximate evolution of $X$}. If it satisfies the estimate for every $t\in (0,c_0R^2]$, we say it is a ($\Lambda$) \textbf{physical solution}. 
\end{definition}

The proof of Theorem \ref{main_thm2} is by iteration; The idea is the following: Interpolating the Hausdorff bounds and the curvature bounds, we get that $X^r$ is locally a graph of a function $u$  with a small gradient over the the hyperplane approximating $X$ at scale $r$. Letting $X^r$ flow for a short yet substantial time (compared to $r^2$), we will be able to extend those $C^0$ and $C^2$ estimates and interpolate again to provide a gradient bound which will be, say, $1000000$ times bigger than the initial gradient bound. Using a new interior estimate for the graphical MCF (Theorem \ref{main_estimate}) we will get an improved bound on the second fundamental form for the evolved hypersurface. By bounding the displacement of $X^r_t$ from $X^r$ and by the Reifenberg property of $X$, we will see that $X^r_t$  can serve as a good candidate for $X^{\theta r}$ for some fixed $\theta>1$. This will allow us to iterate.  Most of the above strategy is carried out in Lemma \ref{main_lemma}. 

\vspace{5 mm}

Obtaining the desired improved curvature bound, while well expected,  is not a trivial task. The existing estimates of Ecker and Huisken (\cite{EH}, see also \ref{EH_curv}) are not good enough when the \textit{norm of the gradient} $|\nabla u|$ is small, as they depend on the so called \textit{gradient function} $\sqrt{1+|\nabla u|^2}$. There are several approaches to deriving sharp estimates in our situation. It turns out that the optimal result using only the norms of the gradient is not sufficiently good to perform iteration even in dimension one. Indeed, even local regularity, which in the graphical case reduces to a weighted $L^1$ estimate, does not give a sufficient  bound in high dimensions. The standard Schauder estimate, using the $C^0$ norm of $u$ in a cylinder turns out to give an insufficient estimate too, since a priori, $u$ remains comparable to its initial value for too little time. Instead, thinking of the graphical mean curvature flow as  a non-homogeneous heat equation on a thick cylinder, we will have to use both the $C^0$ norm at the initial time slice to control the contribution from the initial data and the $C^1$ bound in the cylinder to control the contribution from the boundary and the effect of the non-linearity. As described above, the gradient bound in the cylinder will be, say, $1000000$ times bigger than at the beginning, but as the contribution of the boundary is so weak and as the non-linearity is so marginal (for small gradients, after a short while) this will not matter. The resulting estimate, Theorem \ref{main_estimate}, is somewhat technical and will therefore not be stated precisely in the introduction. 

\vspace{5 mm}

Once Theorem \ref{main_thm2} is established the existence part of Theorem \ref{real_main_thm} follows immediately from compactness. It is easy to see that the limiting flow will actually be a weak set flow in the sense of Definition \ref{ls_flow} and that it will satisfy the physicality assumption of Definition \ref{app_phys_sol}.
\vspace{5 mm}

The main ingredient in the proof of the uniqueness of the flow is the following separation estimate. 
\begin{theorem}[Separation estimate]\label{main_thm3} 
There exist $\Lambda>0$ and $C>0$ such that if $Y^r_t$ is an $r$-approximate evolution of $X$ and $Z^s_t$ is an $s$-approximate evolution of $X$ with $s\leq r$  (with $c_0(\Lambda),\ldots, c_3(\Lambda)$ from Theorem \ref{main_thm2}), then for every $t\in [c_3r^2,c_0R^2]$
\begin{equation}
d_H(Y^r_t,Z^s_t) \leq Cr^{1/2}t^{1/4}.
\end{equation}
\end{theorem}
\noindent The following partial uniqueness result is an immediate corollary of Theorem \ref{main_thm3} and Theorem \ref{main_thm2}.
\begin{corollary}\label{weak_uniq_cor}
With the choice of $\Lambda$ of Theorem \ref{main_thm3} and $\de(\Lambda),c_0(\Lambda),\ldots, c_3(\Lambda)$ of Theorem \ref{main_thm2}, if $X$ is $(\de,R)$-Reifenberg flat and $X^r_t$ are the outputs of Theorem \ref{main_thm2},  the full limit
\begin{equation}
\lim_{r\rightarrow 0}X^r_t
\end{equation}
exists and is in fact the unique physical evolution of $X$. 
\end{corollary}

The idea of the proof of Theorem \ref{main_thm3} is that the conditions of Definition \ref{app_phys_sol} imply that for fixed $t>0$ and  small $r,s$, $Z^{s}_t$ is a graph of a function $u$ over $Y^{r}_t$ in a tubular neighborhood of the latter. In Lemma \ref{graph_pde} we derive a PDE for $u$, a derivation which is the parabolic analogue of (and is based on) a similar calculation done for the minimal surface case in \cite{CM}. Once this is done, a bootstrap argument based on the maximum principle, starting from  the ``crude'' bounds coming from \eqref{uni_est_dist}, will give an estimate for $d_H(Y^{s}_t,Y^r_t)$.

\vspace{5 mm}

While Theorem \ref{real_main_thm} is much stronger than Corollary \ref{weak_uniq_cor}, concluding it using what we described already is very easy. The only thing one needs to note is that in Theorem \ref{aux_reif_thm}, the approximating hypersurfaces can be chosen to be either entirely in the compact domain bounded by $X$ or in its complement (see Corollary \ref{aux_reif_cor}). Those inward and outward $r$-approximate evolutions of $X$ form barriers to the level set flow from inside and outside. By Theorem \ref{main_thm3} they converge to the same thing, from which Theorem \ref{real_main_thm} follows.  

\vspace{5 mm}

The organization of the paper is as follows: In Section \ref{prelim} we sketch the proof of Theorem \ref{aux_reif_thm} and collect some more auxiliary result. In Section \ref{exist_sec} we prove Theorem \ref{main_thm2} assuming the interior estimate, Theorem \ref{main_estimate}. In Section \ref{uniq_sec} we prove Theorem \ref{main_thm3} and conclude the proofs of Theorem \ref{real_main_thm}. In Section \ref{est_sec} we close the argument by proving Theorem  \ref{main_estimate}.

\vspace{5 mm}

\noindent\textbf{Acknowledgments:} The author wishes to thank his advisor Bruce Kleiner for suggesting the problem, as well as for many useful discussions. He also wishes to thank Jacobus Portegies for carefully reading and commenting on an earlier version of this note, as well as for his many suggestions. He also wishes to thank  Robert Haslhofer, for his many suggestions and comments about an earlier version of this notes, and Jeff Cheeger for his generous support during the months in which this project was initiated.

\section{Preliminaries}\label{prelim}
In this section we record some known theorems and derive several simple auxiliary results which will be used later. 
\vspace{5 mm}

We first remark on the proof of  Theorem \ref{aux_reif_thm}, as it appears in \cite{HW}. The first reason we do not regard this theorem as a black box is that condition (3) of Theorem \ref{aux_reif_thm} does not appear in \cite{HW} explicitly. It is, however, a transparent corollary of their construction. The second reason is that we will need to generalize it a bit to force the approximating surfaces to be entirely inside or outside $X$ (see Corollary \ref{aux_reif_cor}). Let $X$ be an $(\de,R)$-Reifenberg flat set, and let $\Omega$ be the domain bounded by $X$ (recall that by \cite{Reif} $X$ is a topological manifold, and so by Jordan's separation theorem there exists such a domain).

\proof[Proof sketch of Theorem \ref{aux_reif_thm} \cite{HW}] The approximating surfaces $X^r$ are constructed as  level sets of  mollifications at scale $r$ of the characteristic function of the domain bounded by $X$. More precisely, choose a smooth radial bump function $\tilde{\phi}\in C^\infty_0(B(0,1))$  such that:
\begin{enumerate}
\item $0 \leq \tilde{\phi} \leq 1$
\item $\tilde{\phi}|_{B(0,1/2)}= 1$ 
\end{enumerate}
and let $c_1$ be a constant such $\phi=c_1\tilde{\phi}$ satisfies $\int \phi(x)dx=1$. Define 
\begin{equation}
\phi_r(x)=\frac{1}{r^n}\phi\left(\frac{x}{r}\right).
\end{equation}
Now, since $X$ is in particular a topological manifold (as was shown in \cite{Reif}), by Jordan separation theorem we can set $\chi$ to be the characteristic function of the domain $\Omega$ bounded by $X$ and let $\chi_r(x)=\phi_r \star \chi(x)$. The approximating sets $X^r$ are defined to be
\begin{equation}
X^r=\chi_r^{-1}\left(\frac{1}{2}\right).
\end{equation}

\noindent (1) To show that $X^r$ satisfies the first condition of Theorem \ref{aux_reif_thm} for $r<R/4$, choose $x^r\in X^r$, $x\in X$ which is closest to $x^r$, $T$ a hyperplane that satisfies the Reifenberg condition at $x$ at scale $4r$ and $\nu$ a normal to $T$ such that 
\begin{equation}\label{rd1}
\Omega\cap B(x,4r) \supseteq \{y\;|\;\left<y-x,\nu\right> \geq 4r\de\}\cap B(x,4r)
\end{equation}
and
\begin{equation}\label{rd2}
\Omega^c\cap B(x,4r) \supseteq \{y\;|\;\left<y-x,\nu\right> \leq -4r\de\}\cap B(x,4r).
\end{equation}

By the definition of $X^r$, it is clear that $B(x^r,r) \subseteq B(x,4r)$ and \eqref{rd1},\eqref{rd2} imply that $\left|\left<x^r-x,\nu\right>\right| \leq 4r\de$, else $\chi_r(x^r)$ would be too low/high. This also implies $d(x^r,X)<8r\de$. Similarly for any $x\in X$ and $T,\nu$ as above $\chi_r(x+4r\de \nu) > \frac{1}{2}$ while $\chi_r(x-4r\de \nu) < \frac{1}{2}$ so by the intermediate value theorem we also have $d(x,X^r) \leq 4r\de$. Thus
\begin{equation}
d_H(X^r,X) \leq 8\de r.
\end{equation}

\noindent (2) To show that $X^r$ satisfies the second condition of Theorem \ref{aux_reif_thm}, choose $x,T,\nu$ as above and  let $\{e_1,\ldots,e_n\}$ be an orthonormal basis for $T$. Setting $S=\{y\;|\;\left|\left<z-x,\nu\right>\right| \leq 4r\de\}$ and taking $y\in B(x,2r)\cap S$ one can compute, splitting the convolution integrals according to $S$ and the two half spaces from \eqref{rd1},\eqref{rd2} (see \cite[2.2]{HW}), that if $\de$ is sufficiently small
\begin{equation}\label{der_bound}
\begin{aligned}
&\;\;\;\;\;\;\;\;\;\;\;\;\;\;\;\;\;\;\;\left<\nabla \chi_r(y),\nu\right> \geq \frac{c_1}{r},\;\;\;\;\;\; \left|\left<\nabla \chi_r(y),e_i\right>\right| \leq \frac{c_2\de}{r}, \\
&\left|\mathrm{Hess}\chi_r(\nu,\nu)\right| \leq \frac{c_2}{r^2},\;\;\; \left|\mathrm{Hess}\chi_r(e_i,\nu)\right| \leq \frac{c_2\de}{r^2},\;\;\; \left|\mathrm{Hess}\chi_r(e_i,e_j)\right| \leq \frac{c_2\de}{r^2},
\end{aligned}
\end{equation}
for some constants $c_1,c_2$. Defining $C(x,r)=\{y\; | \; \left|T(y-x)\right|\leq r\}$, the first inequality of \eqref{der_bound} together with what was done in (1) shows that in $B(x,2r)\cap C(x,r)$, $X^r$ is a graph of a function $u$ over $T$ (thinking of $T$ as an affine hyper-plane passing through $x$) with 
\begin{equation}\label{close_dist}
|u|\leq 4r\de.
\end{equation}
 The other inequalities of \eqref{der_bound} further imply that
\begin{equation}\label{small_curva}
|A| \leq \frac{c_3\de}{r}.
\end{equation}    

\noindent (3) Note that condition \eqref{small_curva} together with $X^r$ being a graph over $T$ in $B(x,2r)\cap C(x,r)$ imply that for $\de>0$ small enough $\mathrm{inj}(X^r)\geq \frac{r}{2}$, where $\mathrm{inj}$ denotes the injectivity radius of the normal exponential map. Now,  considering the scale $\frac{r}{4}$, if $\de$ is small enough, the same $T$ as the one for scale $r$ will work. By using parts (1) and (2) for $X^{r/4}$ and in particular, it being a graph over $T$ at scale $r/4$, one sees by interpolating \eqref{close_dist} and \eqref{small_curva} that $X^{r/4}$ is a graph of a function $u$ over $X^r$ with $|u(y)| \leq 6r\de$ for every $y\in X^r$ (see also Lemma \ref{interpolation2}). To put differently, there is a homeomorphism $\phi_r:X^r \rightarrow X^{r/4}$ with 
\begin{equation}
d(\phi_r(y),y) \leq 6r\de.
\end{equation}
Composing those maps we see that there is a homeomorphism $\phi_{r}^{k}:X^r \rightarrow X^{r/4^k}$ with 
\begin{equation}\label{tot_diff}
d(\phi_r^k(y),y) \leq 12r\de.
\end{equation}
Now, note that the analysis in (1),(2) shows that for every $x\in X$ one can write  $X^r\cap B(x,r)=G\cup B$ where $G$ is connected and $B \cap B(x,(1-8\de)r)=\emptyset$. Thus, using \eqref{tot_diff} we see that one can write $X^{r/4^k}\cap B(x,r)=G\cup B$ where $G$ is connected and $B \cap B(x,(1-20\de)r)=\emptyset$. This concludes the (sketch of the) proof of the theorem $\Box$.

\vspace{5 mm}

The following slight strengthening of Theorem \ref{aux_reif_thm} will play a role in the final stage of our argument for uniqueness.

\begin{corollary}\label{aux_reif_cor}
There exist some constants $c_1,c_2>0$ such that if $X$ is $(\de,R)$-Reifenberg flat for $0<\de<\de_0$ and bounds a domain $\Omega$  then there exists a family of surfaces $(X^r_\pm)_{0<r<R/4}$ such that
\begin{enumerate}
\item  $d_H(X^r_\pm,X) \leq c_1\de r$.
\item $|A(x)| \leq \frac{c_2\de}{r}$ for every $x \in X^r_\pm$.
\item for every $x\in X$, $r\in(0,R/4)$ and $s\in(r,R/4)$, $B(x,s)\cap X^r_\pm$ can be decomposed as
\begin{equation}
B(x,s)\cap X^r_\pm=G\cup B
\end{equation}
where $G$ is connected and $B \cap B(x,(1-40\de)s)=\emptyset$. 
\item $X^r_- \subseteq \Omega$ and $X^r_+ \subseteq\bar{\Omega}^c$.
\end{enumerate}
\end{corollary} 
\proof  Let $N$ be the exterior unit normal to the $X^r$ from Theorem \ref{aux_reif_thm}. When $\de$ is small enough, conditions (1),(3) of Theorem \ref{aux_reif_thm} imply that  $X^r$ has a tubular neighborhood of thickness $r/4$ (see Lemma \ref{taub_lemma} for a proof in an analogous situation). Moreover, computing the third partials for $\chi_r$ from the above theorem shows that the $X^r$ also satisfy the estimate
\begin{equation}
|\nabla A| \leq c_3\frac{\de}{r^2}.
\end{equation}
Thus, considering
\begin{equation}
X^r_{\pm}=\{x\pm 10\de r N(x)\;|\;x\in X^r\}
\end{equation}
we see that $X^r_\pm$ satisfy condition (1)-(3). Let $x\in X$ and let $T,\nu$ be as in (1) in the proof of Theorem \ref{aux_reif_thm}, then since for every in $y\in B(x,r)\cap X^r$ the tangent space to $X^r$ at $y$ is almost parallel to $T$ by (2) in the proof, we see by part (1) of the proof that one of the $X^r_{\pm}\cap B(x,r)$ will lie in $\Omega$ and the other will lie in $\bar{\Omega}^c$. Thus $X^r_{\pm}\cap X=\emptyset$ with one of the $X^r_{\pm}$ in each component $\Box$.

\vspace{5 mm}
We will now recall  the gradient estimate of Ecker and Huisken. Let $X$ be a hypersurface in  $\mathbb{R}^{n+1}$ such that in $B(x,r)$ it can be parametrized locally as a graph of a function $u$ over the first $n$ coordinates. If $\nu(x)$ is the unit normal to $X$ at $x$, the gradient function $v:X\cap B(x,r)\rightarrow \mathbb{R}_{+}$  is defined to be
\begin{equation}
v(x)=\frac{1}{\nu \cdot e_{n+1}}=\sqrt{1+|\nabla u|^2}.
\end{equation}
\begin{theorem}[Ecker-Huisken Gradient estimate \cite{EH}, see also \cite{Eck}]\label{EH} 
Let $(X_{t})_{t\in(t_0,t_1)}$ be a solution for the mean curvature flow in $\mathbb{R}^{n+1}$ and suppose that in the ball $B(x_0,r)$ $X_{t_0}$ is locally a graph over the first $n$ coordinates. Then $X_t$ is locally a graph over the first $n$ co-ordinates in $B(x_0,\sqrt{r^2-2n(t-t_0)})$ and 
\begin{equation}
\left(1-\frac{|x-x_0|^2+2n(t-t_0)}{r^2}\right)v(x,t) \leq \max_{x\in B(x_0,r)\cap X_{t_0}} v(x,t_0)
\end{equation}
$\Box$.
\end{theorem}

Controlling the gradient in a space-time neighborhood allows one to control the second fundamental form (and its derivatives) if one allows the hypersurface to evolve a little bit.
\begin{theorem}[\cite{EH}, see also \cite{Eck}]\label{EH_curv}
There exist a universal constant $C=C(n)$ such that if $v \leq v_0$ in $B(x,r)\times(t_0,t_0+r^2)$ then
\begin{equation}
|A|^2 \leq C\frac{v_0^4}{r^2},\;\;\;\;\;\; |\nabla A|^2 \leq C\frac{v_0^4}{r^4},\;\;\;\;\;\;\;|\partial_tA|^2 \leq C\frac{v_0^4}{r^6} 
\end{equation}
in $B(x,r/2)\times(t_0+3r^2/4,t_0+r^2)$ $\Box$.
\end{theorem}

The above classical curvature estimate of Ecker and Huisken will not be good enough for our purposes. We will be interested in considering the case where $|\nabla u|$ is very small, in which the above estimate is not very useful. Indeed, since $v=\sqrt{1+|\nabla u|^2}$ the above estimate will yield that $|A|^2$ is, up to a constant, smaller than  $\frac{1}{r^2}$, even if $|\nabla u|$ were very small (in the entire neighborhood). When $|\nabla u|=0$ however, $|A|=0$ so there is clearly a big gap to be filled in that regime. The entire Section \ref{est_sec}, culminating in the proof of Theorem \ref{main_estimate}, will deal  with proving a curvature estimate suitable for small $|\nabla u|$ (and small initial $|u|$). 

\vspace{5 mm}

Controlling the second fundamental form at a certain time allows one to control it for a little bit.
\begin{lemma}\label{curv_max_prin}
For every $\delta>0$ there exists $c=c(n,\delta)>0$ such that if $X_t$ flows by mean curvature and 
\begin{equation} |A(0)| \leq \alpha
\end{equation}
 then
\begin{equation}
|A(t)|\leq (1+\delta)\alpha
\end{equation}
for $0 \leq t \leq \frac{c}{\alpha^2}$.
\end{lemma}  
\proof Since under the mean curvature flow
\begin{equation}
\frac{d}{dt}|A|^2 =\Delta|A|^2-2|\nabla A|^2+2|A|^4 \leq \Delta |A|^2+2|A|^4
\end{equation}
(see \cite{Man} for instance) by the maximum principle we obtain that 
\begin{equation}
|A(t)| \leq \frac{\alpha}{\sqrt{1-2\alpha^2t}}
\end{equation}
for as long as the denominator does not vanish. The result follows $\Box$.

We will conclude this section with the following simple geometric lemma for the interpolation of Hausdorff bounds and curvature bounds.

\begin{lemma}[Interpolation]\label{interpolation}
For every $\delta>0$ and $\alpha>0$  there exists $\beta_0(\alpha,\delta)>0$ such that for every $\beta<\beta_0$ the following holds: Assume $p\in X$ where $X$ is a hypersurface such that in $B(p,r)$ we have
\begin{enumerate}
\item $|A| \leq \frac{\alpha}{r}$ .
\item $d_H(P\cap B^n(p,r),X \cap B^n(p,r)) \leq \beta r$ for $P=\spa\{e_1,\ldots,e_n\}$.
\end{enumerate}
 Then inside $B(p,r)\cap \left(B^n(p,(1-\delta)r)\times \mathbb{R}\right)$, the connected component of $p$ is a graph of a function $u$ over $P$  and we have the estimate
\begin{equation}
|\nabla{u}| \leq \frac{\sqrt{2\beta\alpha -\alpha^2\beta^2}}{1-\alpha\beta} \cong \sqrt{2\beta\alpha}
\end{equation}
(and $|u|\leq \beta r$).
\end{lemma}
\begin{proof}
Assume w.l.g. $r=1$ and $p=0$ and denote $Q=P^\perp$ and $C_{\delta,\beta}=B^n\left(0,(1-\delta)\right)\times [-\beta ,\beta ]$. For $\beta$ sufficiently small $C_{\delta/4,\beta}\subseteq B(0,1)$ and $\alpha\beta<1$. Now, let $x\in C_{\delta/2,\beta}$ and let $\gamma(t)$ be a unit speed geodesic with $\gamma(0)=p$. We may assume w.l.g., by possibly changing the parametrization according to  $t\mapsto -t$, that $\left<\gamma'(0),e_{n+1}\right>=\max_{v\in Q,\;||v||=1}\left<\gamma'(0),v\right>$ and that $x_{n+1}(\gamma(t)) \geq 0$. Letting $f(t)=x_{n+1}(\gamma(t))$ we find $f'(t)=\left<\gamma'(t),e_{n+1}\right>$ and $f''(t)=\left<\gamma''(t),e_{n+1}\right>=\left<\gamma''(t),e_{n+1}-\left<\gamma'(t),e_{n+1}\right>\gamma'(t)\right>\geq -\alpha\sqrt{1-f'(t)^2}$. The equality case of the above ODE for $f'(t)$ corresponds to a circle of radius $\frac{1}{\alpha}$. Letting $\mu(t):\mathbb{R}\rightarrow \mathbb{R}^2$ be a clockwise and unit speed parametrized circle of radius $\frac{1}{\alpha}$ with $\mu(0)=(0,0)$ and $\left<\mu'(0),e_2\right>=f'(0)$ we see that as long as $x_2(\mu(t))$ is increasing, and as long as $\gamma(t)\in C_{\delta/4,\beta}$,  $x_{n+1}(\gamma(t))\geq x_2(\mu(t))$. For $\beta$ sufficiently small (depending on $\alpha,\delta$) $x_2(\mu(t))$ will reach its maximum at time $0<T<\delta/4$ so the extra condition $\gamma(t)\in C_{\delta/4,\beta}$ is redundant. Thus $x_2(\mu(t))\leq x_{n+1}(\gamma(t))\leq \beta$, and an easy calculation for circles in the plane gives the bound
\begin{equation}\label{almost_tan}
\tan\angle(T_xX,P) \leq \frac{\sqrt{2\beta\alpha -\alpha^2\beta^2}}{1-\alpha\beta}
\end{equation}
for $\beta$ sufficiently small. 

What remains to be shown is that the connected component of $p$ is indeed a graph. Assume there exist $x_1,x_2\in X\cap C_{\delta,\beta}$ with $x_1\neq x_2$ but $P(x_1)=P(x_2)$ (where we use $P$ both for the hyperplane and for the projection operator to it). Observe that by \eqref{almost_tan}, $X\cap \overline{C_{\delta,\beta}}$ is a sub-manifold with boundary. Let $\gamma:[0,a]\rightarrow X \cap \overline{C_{\delta,\beta}}$ be a minimizing geodesic between $x_1$ and $x_2$. Such a geodesic is always $C^1$ and is smooth for as long $\gamma(t)$ is away from the boundary. For such $t$ however $||P(\gamma''(t))|| \leq \sqrt{3\alpha\beta}\alpha$ by \eqref{almost_tan} and so for $\beta$ sufficiently small, and as $\gamma'(0)$ is almost parallel to $P$,  $P(\gamma(t)))$ is almost a straight line until it hits the boundary (at some $t<4$). Since $\gamma(t)$ is $C^1$, and intersects the boundary with an exterior normal component, this is a contradiction. 

To see that for every $y\in B^n(0,1-\delta)$ there is some $x\in X$ with $P(x)=y$, note that by the Hausdorff condition, we can find $\bar{x}\in X\cap B(0,(1-\delta/2))$ with $d(\bar{x},y)\leq \beta$ (when $\beta$ is small). Taking $\bar{y}=P(\bar{x})$ we see, again, by \eqref{almost_tan} for $\bar{x}$, and the fact that the curvature scale $\frac{1}{\alpha}$ is far bigger than $\beta$, that there will exist a point over $y$ as well.            
\end{proof}

\section{Uniform Estimates and Existence of Smooth Evolution}\label{exist_sec}
The purpose of this section it to prove Theorem \ref{main_thm2} which will immediately imply the existence part of Theorem \ref{real_main_thm}, i.e. that at least one weak set flow (see Definition \ref{ls_flow}) of an $(\de,R)$-Reifenberg flat set is smooth whenever $\de>0$ is small enough. In Section \ref{interior_formulation_sec} we will state the interior estimate we will employ and remark on why is it plausible. The proof will be deferred to Section \ref{est_sec}.  In Section \ref{iter_sec} we will perform the iteration step that was described in the introduction. In Section \ref{uniform_sec} we will  prove Theorem \ref{main_thm2} and derive the existence of a smooth weak set flow.
\subsection{Interior Estimate for Graphical Mean Curvature Flow}\label{interior_formulation_sec}

\noindent In order to implement the iteration, it would be the most comfortable to work with the following interior estimate for mean curvature flow, which will be proved in Section \ref{est_sec}. 

\begin{theorem}[main estimate]\label{main_estimate}
There exist some $c\geq 1$ such that for every $\delta>0$ and $M>0$, there exist positive $\tau_0=\tau_0(M,\delta)<<1$ and $\lambda_0=\lambda_0(M,\delta)<<1$ such that for every $0<\lambda<\lambda_0$ there exists some  $\de_0=\de_0(\delta,M,\lambda)$ such that for every $0<\tau<\tau_0$  and $\de<\de_0$ the following holds:

\noindent If $u:B(p,r)\times [0,\tau r^2]\rightarrow \mathbb{R}$ is a graph moving by mean curvature such that:
\begin{enumerate}
\item For every $(x,t) \in B(p,r)\times [0,\tau r^2]$ 
\begin{equation}
|\nabla u(x,t)| \leq M\de. 
\end{equation}
\item  For every $(x,t) \in B(p,r)\times [0,\lambda \tau r^2]$ 
\begin{equation}
|\nabla u(x,t)| \leq \de.
\end{equation}
\item For every $(x,t) \in B(p,r)\times [0,\tau r^2]$ 
\begin{equation}
|u(x,t)| \leq M^2\beta r. 
\end{equation}
\item  For every $(x,t) \in B(p,r)\times [0,\lambda \tau r^2]$ 
\begin{equation}
|u(x,t)| \leq \beta r.
\end{equation}
\end{enumerate}
Then for every $(x,t)\in B(p,(1-\delta)r) \times [0,\tau r^2]$
\begin{enumerate}[label=\Alph*.]
\item 
\begin{equation} 
|A(p,t)| \leq (1+\delta)\frac{1}{\sqrt{\pi}}\frac{\de}{\sqrt{t}}.
\end{equation}
\item 
\begin{equation}
|A(p,t)| \leq c\frac{\beta r}{t}+\delta\frac{\de}{\sqrt{t}}
\end{equation}
\end{enumerate}
\end{theorem}

\begin{remark}
The statement of the above estimate may appear somewhat confusing. We hope that the following trailer to Section \ref{est_sec}  will make it appear more plausible. There is nothing wrong with fixing  $r=1$.
\begin{enumerate}
\item When $|\nabla u|$ is very small, we are dealing essentially with the heat equation, as the non linearity is very small. The estimate 
\begin{equation}
|\nabla u(x,t)| \leq \frac{1}{\sqrt{\pi}}\frac{||u(-,0)||_\infty}{\sqrt{t}}
\end{equation}
 is what one gets when estimating the first derivative to the physical solution of the heat equation in the full space at time $t$ in terms of the $\sup$ norm of the initial time slice. Estimate (A) reflects that fact,  as the derivative of a solution to the heat equation satisfies the heat equation itself. The first term in estimate (B) reflects a similar bound on the second derivative.  

\item Since we are dealing with a domain with boundary, one can not expect to get the same estimate as for the entire space. However, for $\tau$ very small, from the perspective of the point $(0,\tau)$, the $0$ time  slice - a ball of radius $1$ - will look like the entire space. Therefore we should get a constant very close to $\frac{1}{\sqrt{\pi}}$, as expressed in (A). The parameter $M$ in the estimate makes sure that the contribution from the boundary doesn't change the result by too much.  A similar reasoning leads to the first term of (B). 

\item Since the equation is non-linear, there should also be a term coming from the non-linearity. Since the non-linearity is quadratic in the gradient, when the bound $|\nabla u|$ is small enough, it will contribute as little as any small fraction of that gradient bound.  
\end{enumerate} 
\end{remark}

\subsection{Iteration}\label{iter_sec}

Before we dive into the iteration lemma, we need the following two similar calculations. The first allows one to extend the curvature and Hausdorff bounds for a short time.

\begin{lemma}\label{short_time_mot_bound}
For every $\delta>0$ there exists some $\tilde{c}=\tilde{c}(n,\delta)>0$ with the following property: Assume $\alpha,\beta$ are such that $\alpha\beta<1$ and $Y$ is a smooth hypersurface such that for some $r>0$:
\begin{enumerate}
\item For every $y\in Y$ there exists a hyperplane $P_y$ such that
\begin{equation}
d_H(B(y,r)\cap P_y,B(y,r)\cap Y) \leq \beta r.
\end{equation}
\item For every $y\in Y$
\begin{equation}
|A(y)| \leq \frac{\alpha}{r}.
\end{equation}
\end{enumerate}
Then denoting by $Y_t$ the mean curvature flow emanating from $Y$ and writing $E=2\alpha\beta$ we have for every $0\leq t \leq   \tilde{c}\frac{1}{\alpha^2}E r^2$:
\begin{enumerate}
\item $|A(t)| \leq (1+\delta)\frac{\alpha}{r}$.
\item $d_H(B(p,(1-10^{-5}\delta\beta)r)\cap P ,B(p,(1-10^{-5}\delta\beta)r) \cap Y_t) \leq (1+\delta)\beta r $.
\end{enumerate}  
\end{lemma}
\proof Assume $r=1$. Using the global curvature bound $\alpha$, by Lemma \ref{curv_max_prin} we can find $c_1=c_1(n,\delta)$ such that if 
\begin{equation}
0 \leq t\leq c_1\frac{1}{\alpha^2}
\end{equation}
 we have
\begin{equation}
|A(t)| \leq (1+\delta/4)\alpha.
\end{equation}  
Using the curvature bound to estimate the motion, we see that there exists $c_2(n,\delta)<c_1$ such that for $t\leq c_2\frac{\beta}{\alpha}$ the surface moves by at most $10^{-6}\delta\beta$. Since $\alpha\beta<1$, for every
\begin{equation}
0 \leq t \leq c_2\min\{\frac{\beta}{\alpha},\frac{1}{\alpha^2}\}=c_2\frac{1}{\alpha^2}2\alpha\beta
\end{equation}
we have:
\begin{enumerate}
\item $|A(t)| \leq (1+\delta)\alpha$
\item $d_H(B(p,(1-10^{-5}\delta\beta))\cap P ,B(p,(1-10^{-5}\delta\beta)) \cap Y_t) \leq (1+\delta)\beta $.
\end{enumerate}
(See also Lemma \ref{trian_ball} for how motion bounds are used to obtain Reifenberg flatness at a certain scale) $\Box$.

The second calculation allows one to extend gradient estimates for longer times, gaining a definite large multiplicative error. 
\begin{lemma}\label{time_scales}
For every $\delta>0$ and $\alpha>1$ there exists $M=M(n,\alpha)>0$, $\beta_0=\beta_0(n,\alpha,\delta)>0$ such that for every $0<\beta<\beta_0$, if we set $E=2\alpha\beta$ and $T=Er^2$ the following holds: Assume $(Y_t)_{t\in[0,T]}$ is a mean curvature flow with  
\begin{equation}
|A(0)|\leq \frac{\alpha}{r}
\end{equation} 
such that for every $y\in Y_0$ there exists a hyperplane $P_y$ such that
\begin{equation}
d_H(B(y,r)\cap P_y,B(y,r)\cap Y_0) \leq \beta r.
\end{equation}
Then setting $(\phi_t)_{t\in [0,T]}$ to be the parametrized flow starting from $Y$ (so $Y_t=\phi_t(Y)$) and assuming w.l.g that $y=0$ and $P_y=\spa\{e_1,\ldots,e_n\}$, the connected components of $\phi_t(y)$ in $Y_t\cap B(y,r)\cap \left( B^n(y,(1-\delta)r)\times \mathbb{R}\right)$ is a graph of a function $u:B^{n}(y,(1-\delta)r)\times [0,T]\rightarrow \mathbb{R}$ flowing by mean curvature.
Moreover, we have the estimates
\begin{equation}
|\nabla u(z,t)|\leq M\sqrt{E},  \;\;\;\;\; |u(z,t)| \leq M^2\beta r
\end{equation}
for every $z\in  B^{n}(y,(1-\delta)r)$ and $0 \leq t \leq T$.
Letting $\tilde{c}$ be the constant from Lemma \ref{short_time_mot_bound}, for $z\in  B^{n}(y,(1-\delta)r)$ and $0 \leq t \leq \frac{\tilde{c}}{\alpha^2}T$ we have the better estimate: 
\begin{equation}
|\nabla u(z,t)|\leq 2\sqrt{E},  \;\;\;\;\; |u(z,t)| \leq 2\beta r.
\end{equation}
\end{lemma}

\proof Assume w.l.g. that $r=1$. Since $T=2\alpha\beta$, according to Lemma \ref{curv_max_prin} we have for $\beta<\beta_0(\alpha)$
\begin{equation}
|A(t)| \leq 2\alpha
\end{equation}   
for every $0\leq t \leq T$. Thus, we can bound the motion to obtain 
\begin{equation}
d(\phi_t(z),\phi_0(z)) \leq 4\sqrt{n}\alpha^2\beta
\end{equation}
for every $0\leq t \leq T$ and $z\in Y$. Now, when $\beta<\beta_0$ this displacement is very small so as before (see also Lemma \ref{trian_ball}) we get $5\sqrt{n}\alpha^2\beta$ closeness to planes on a slightly smaller ball. We can now use Lemma \ref{interpolation} to obtain graphicallity for every $0 \leq t \leq T$, as well as the estimates
\begin{equation}
|\nabla u(z,t)| \leq 2\sqrt{2\alpha\cdot 5\sqrt{n}\alpha^2\beta}=M(\alpha)\sqrt{E}
\end{equation}     
and 
\begin{equation}
|u(z,t)| \leq 5\sqrt{n}\alpha^2\beta\leq M(\alpha)^2 \beta.
\end{equation}
The second estimate now follows from lemma \ref{short_time_mot_bound} $\Box$.

\vspace{5 mm}

\noindent The main technical point of the argument  is carried out in the following lemma.
\begin{lemma}\label{main_lemma}
 For every $\Lambda>0$ there exist some $\alpha,\beta>0$ and a ``scale change'' parameter $\theta>20$ such that setting $E=2\alpha\beta$ and $T=Er^2$ the following holds: 

\noindent Assume $Y$ is a smooth hypersurface such that for some $r>0$:
\begin{enumerate}
\item for every $y\in Y$ there exists a hyperplane $P_y$ such that
\begin{equation}
d_H(B(y,r)\cap P_y,B(y,r)\cap Y) \leq \beta r.
\end{equation}
\item for every $y\in Y$
\begin{equation}
|A(y)| \leq \frac{\alpha}{r}.
\end{equation}
\end{enumerate}
Then $Y$ flows smoothly by mean curvature for time $T$. Moreover, setting $(\phi_t)_{t\in [0,T]}$ to be the parametrized flow starting from $Y$ and letting $Y_t$ be the flow of hypersurfaces (so $Y_t=\phi_t(Y)$),  we have:
\begin{enumerate}[label=\Alph*.]
\item $d(\phi_0(y),\phi_T(y)) \leq \frac{\beta}{20}(\theta r)$ for every $y\in Y$ and in particular $d_H(Y,Y_T) \leq \frac{\beta}{20}(\theta r)$.
\item for every $y\in Y_T$
\begin{equation}
|A(y)| \leq \frac{\alpha}{\theta r}.
\end{equation}
\end{enumerate} 
Moreover, we can choose the parameters in such a way that the following relations hold:
\begin{enumerate}[label=\Roman*.]
\item $2\theta E<\min\{10^{-6},\Lambda^{-1}\}$.
\item $\alpha>100$.
\item $\alpha^3\beta<\frac{1}{640}$.
\end{enumerate}
\end{lemma}

\proof We want to set things up in a way that will allow us to use Theorem \ref{main_estimate}. Assume w.l.g $r=1$, fix $p \in Y$ and assume w.l.g. that $P_p=\spa\{e_1,\ldots, e_n\}$. 

\begin{enumerate}[label= Step \arabic*.]
\item \underline{Choice of parameters:}
\begin{enumerate}
\item \underline{Choosing $\delta$}: Take $\delta=\frac{1}{800\sqrt{n}}$.

\item\label{alpha_choice} \underline{Choosing $\alpha$:} Choose $\alpha=400\sqrt{n}c$ where $c$ is from Theorem \ref{main_estimate}.

\item \underline{Choosing $M$:}  Choose $M=M(\alpha)$ from Lemma \ref{time_scales}.

\item \underline{Choosing $\tau_0$ and $\lambda_0$:} Having fixed $M,\delta$ we can choose $\tau_0(M,\delta)$ and $\lambda_0(M,\delta)$ as in Theorem \ref{main_estimate}.

\item \underline{Choosing $\lambda$} Let $\tilde{c}(n,\delta)$ be the constant from Lemma \ref{short_time_mot_bound} and choose:
\begin{equation}
\lambda =\lambda(\delta,\alpha,\lambda_0)= \min\{\frac{\tilde{c}}{\alpha^2},\lambda_0\}.
\end{equation}

\item \underline{Choosing $\tau=T$:}  Assuming $E$ is small enough, $T$ can be made to be smaller than $\tau_0$ and so we let $\tau=T$. Note that $E$ is still free, so we have just expressed our wish to set $\tau=E$ without really fixing $\tau$.  

\item \underline{Choosing $\de_0$:} Choose $\de_0=\de_0(M,\delta,\lambda)$ as in the Theorem \ref{main_estimate}.

\item \underline{Choosing $\theta$:} Set $\theta=80000nc>20$.

\item \underline{Choosing $\beta$:} Choose $\beta$ such that the following conditions hold:
\begin{enumerate}
\item $\beta<\beta_0(\alpha,\delta/8)$ of Lemma \ref{interpolation}.
\item $\beta<\beta_0(n,\alpha,\delta/4)$ of Lemma \ref{time_scales}.
\item $\sqrt{E}=\sqrt{2\alpha\beta} < \de_0/2$ (so that Theorem \ref{main_estimate} will hold, see \ref{app_est}).
\item $2\theta E = 4\alpha\beta\theta<\min\{10^{-6},\Lambda^{-1}\}$ (to comply with the statement).
\item $\alpha^3\beta<\frac{1}{640}$ (to comply with the statement).
\end{enumerate}
As all those conditions want $\beta$ to be small, they can be satisfied simultaneously.

\item \underline{Choosing $\de$:} We finally set $\de=\sqrt{E}$. 
\end{enumerate}

\item \underline{Initial bounds:} By the choice of $\beta$, we know that $(\alpha,\beta,\delta/8)$ satisfy the conditions of Lemma \ref{interpolation}. Therefore, we get that  the connected component of $p$ in $Y \cap \left(B^n(p,(1-\delta/4))\times [-\beta ,\beta ]\right)$ is a graph of a function $u$ over $B^n(p,(1-\delta/4))$ with
\begin{equation}\label{lem_est_1}
|\nabla u| \leq \frac{3}{2}\sqrt{2\beta\alpha}=\frac{3}{2}\sqrt{E}.
\end{equation}

\item \underline{Obtaining bounds for long positive times:} By Lemma \ref{time_scales}, if \newline$\beta<\beta_0(n,\alpha,\delta/4)$, $u$ remains a function on  $B^n\left(p,\left(1-\frac{\delta}{4}\right)\right)\times [0,T]$ where the following estimate holds
\begin{equation}
|\nabla u(y,t)| \leq M\sqrt{E}, \;\;\;\;\;\; |u(y,t)| \leq M^2\beta.
\end{equation}

\item \underline{Obtaining bounds for small positive times}: According to the second part of Lemma \ref{time_scales}, for $t\in [0,\lambda T]$, $Y_t$  satisfies the conditions of Lemma \ref{interpolation} and so we obtain that for those times, for  $y\in B^n(p,(1-\delta/2))$
\begin{equation}
|\nabla u(y,t)| \leq 2\sqrt{E}, \;\;\;\;\;\;\;\; |u(y,t)| \leq 2\beta. 
\end{equation}

\item\label{app_est} \underline{Applying Theorem \ref{main_estimate}; estimating the curvature:} By the first estimate of Theorem \ref{main_estimate} we have that for every $t\in[0,E]$ and $y\in Y_t$ 

\begin{equation}\label{good_est}
|A(y,t)|\leq \frac{2\sqrt{E}}{\sqrt{t}}
\end{equation}
and  by the second estimate of Theorem \ref{main_estimate} and by our choice of $\alpha$ and $\delta$, at the final time $T$ we have:
\begin{equation}\label{alph_p_ref}
|A(y,T)| \leq \frac{2c\beta}{E}+2\delta\frac{\sqrt{E}}{\sqrt{E}}=\frac{1}{200\sqrt{n}}:=\alpha'.
\end{equation}

\item\underline{Estimating the motion:}\label{mot_est} The curvature bound \eqref{good_est} implies the mean curvature bound:
\begin{equation}
|H(y,t)| \leq \frac{2\sqrt{n}\sqrt{E}}{\sqrt{t}}
\end{equation}
which can be integrated to obtain the motion bound:

\begin{equation}\label{beta_p_ref}
d(\phi_0(y),\phi_T(y)) \leq 4\sqrt{n}E:=\beta'
\end{equation}

\item \underline{Conclusion:}  Note that by our choice of $\theta$ and by \eqref{alph_p_ref},\eqref{beta_p_ref}
\begin{equation}
|A(y,T)| \leq \alpha' =\frac{\alpha}{\theta} 
\end{equation}
and
\begin{equation}   
\alpha'\beta'=\frac{1}{50}E<\frac{1}{20}\alpha\beta
\end{equation}
which implies
\begin{equation}
d(\phi_0(y),\phi_T(y)) \leq \beta'  \leq \frac{\beta}{20}\theta \;\;\;\;\Box.
\end{equation}
\end{enumerate}

\begin{remark} 
Note that the motion estimate in step 6  illustrates why we needed part (B) of our main estimate, Theorem \ref{main_estimate}. Using the optimal form of part (A) of \ref{main_estimate} up to time $T$, we obtain that, in the language of step 6, for $t\in [0,T]$:
\begin{equation}
|A(y,t)| \leq \frac{1}{\sqrt{\pi}}\frac{\sqrt{E}}{\sqrt{t}}
\end{equation} 
\begin{equation}
|H(y,t)| \leq \sqrt{n}\frac{1}{\sqrt{\pi}}\frac{\sqrt{E}}{\sqrt{t}}
\end{equation}
and thus by integration
\begin{equation}
d_H(Y_0,Y_T) \leq 2\sqrt{n}\frac{1}{\sqrt{\pi}}E.
\end{equation}
Thus, for the corresponding $\alpha',\beta'$ we will necessarily  have $\alpha'\beta'>\alpha\beta$ which will preclude us from iterating, this time with $Y_T$ at scale $\theta$ (see also section \ref{uniform_sec}).
\end{remark}

\subsection{Uniform Estimates}\label{uniform_sec}

Before proving Theorem \ref{main_thm2} we need the following elementary lemma about restricting the triangle inequality to balls.
\begin{lemma}\label{trian_ball}
Suppose $X$ is a $(\beta/10,R)$-Reifenberg set and let $Y$ be a compact set such that for some $r<R$
\begin{equation}
d_H(X,Y) \leq \frac{\beta}{10}r.
\end{equation}
Then for every $y_0\in Y$ there exists a hyperplane $P$ passing through $y$ such that
\begin{equation}
d_H\left(B(y_0,(1-\beta/5)r)\cap P,B(y_0,(1-\beta/5)r)\cap Y\right) \leq \frac{3}{5}\beta r.
\end{equation}
\end{lemma}
\proof Take $x_0\in X$ with 
\begin{equation}d(x_0,y_0) \leq \frac{\beta}{10} r
\end{equation}
and choose a plane $\bar{P}$ passing through $x_0$ such that
\begin{equation}
d_H\left(B(x_0,r)\cap \bar{P},B(x_0,r)\cap X\right) \leq \frac{\beta}{10}r.
\end{equation}
Let $P$ be the plane parallel to $\bar{P}$ passing through $y_0$. Now, taking any \newline$y\in B\left(y_0,(1-\beta/5)r\right)$ and $x$ with \begin{equation}
d(x,y)\leq \frac{\beta}{10} r,
\end{equation} 
we see that $x\in B(x_0,r)$ and so there is a point $\bar{p}\in \bar{P}\cap B(x_0,r)$ with 
\begin{equation}
d(\bar{p},y) \leq \frac{\beta}{5}r.
\end{equation}
Moving a bit inward, this implies that there exist $\bar{p}_1 \in \bar{P}\cap B(x_0,(1-3\beta/10)r)\subseteq B(y_0,(1-\beta/5)r)$ with 
\begin{equation}d(\bar{p}_1,y)\leq \frac{5}{10}\beta r
\end{equation}
 and so taking $p\in P$ closest to $\bar{p}_1$ we obtain $p\in P \cap B(y_0,(1-\beta/5)r)$ with 
\begin{equation}
d(p,y)\leq \frac{3}{5}\beta r.
\end{equation} 

On the other direction, given $p\in P \cap B(y_0,(1-\beta/5)r)$ and $\bar{p}_1 \in \bar{P}\cap B(x_0,(1-\beta/5)r) $ closest to $p$, we can move inward to find $\bar{p}\in  B(x_0,(1-5\beta/10)r)\cap \bar{P}$ with 
\begin{equation}
d(\bar{p},p) \leq \frac{2}{5}\beta r.
\end{equation}
Choosing $x\in X$ with 
\begin{equation}
d(x,\bar{p})\leq \frac{1}{10}\beta r
\end{equation}
 and $y\in Y$ with 
\begin{equation}
d(x,y)\leq \frac{1}{10}\beta r
\end{equation}  
we get 
\begin{equation}
d(p,y) \leq \frac{3}{5}\beta r
\end{equation} 
and $y\in B(x_0,(1-3\beta/10)r) \subseteq B(y_0,(1-\beta/5)r)$ $\Box$.

\vspace{5 mm}      

\noindent We are now in a position to prove Theorem \ref{main_thm2}.
\proof[Proof of Theorem \ref{main_thm2}] Fix $\alpha,\beta,,\theta,E$ as in Lemma \ref{main_lemma}, let $r<R$ and recall that $\theta E<10^{-6}$. Choose $\de>0$ sufficiently small so that:
\begin{enumerate}
\item $c_1\de<\frac{\beta}{10}$.
\item $\de < \frac{\beta}{10}$.
\item $c_2\de < \alpha$.
\end{enumerate} 
Here $c_1,c_2$ are the constants from Theorem \ref{aux_reif_thm}. Letting $Y=X^r$, with the above choice of $\de$ we have:
\begin{enumerate}[label=\Alph*$(r)$.]
\item $|A(y)| \leq \frac{\alpha}{r}\leq \frac{\alpha}{(1-\beta/5)r}$ for every $y\in Y$.
\item $d_H(Y,X) \leq \frac{\beta}{10}r$.
\item By Lemma \ref{trian_ball}, for evert $y\in Y$ there exists a plane $P$ such that
\begin{equation}
d_H(B(y,(1-\beta/5)r)\cap P,B(y,(1-\beta/5) r)\cap Y) \leq \frac{3}{5}\beta r \leq \beta(1-\beta/5)r.
\end{equation}  
\end{enumerate}
Conditions A$(r)$ and C$(r)$ allow us to apply Lemma \ref{main_lemma}, so we can let $Y$ flow smoothly for time $T=E(1-\beta/5)^2r^2$. Moreover, we have:
\begin{enumerate}[label=\Alph*$(\theta r$).]
\item For every $y\in Y_T$
\begin{equation}
|A(y)| \leq  \frac{\alpha}{(1-\beta/5)(\theta r)}. 
\end{equation}
\item For every $y\in Y$ 
\begin{equation}
d(\phi_0(y),\phi_T(y)) \leq  \frac{\beta}{20} (\theta(1-\beta/5) r)
\end{equation}
so by condition $B(r)$ and the fact that $\theta>20$: 
\begin{equation}
d_H(X,Y_T) \leq \frac{\beta}{10}r+\frac{\beta}{20} (\theta(1-\beta/5) r) \leq \frac{\beta}{20}(\theta r)+\frac{\beta}{20} (\theta(1-\beta/5) r) < \frac{\beta}{10}(\theta r).
\end{equation}
\item  If $\theta r <R$, we can use Lemma \ref{trian_ball} to obtain
\begin{equation}
d_H(B(y,(1-\beta/5)\theta r)\cap P ,B(y, (1-\beta/5)\theta r)\cap Y_T) \leq \frac{3}{5}\beta(\theta r) \leq \beta(1-\beta/5)(\theta r).
\end{equation} 
\end{enumerate}
Thus, we conclude that $Y_T$ satisfies the same conditions as $Y$, this time with respect to the scale $\theta r$, so we can restart the process with $Y_T$ instead of $Y$ at scale $\theta r$ instead of $r$ and iterate. The iterations will be performed at times
\begin{equation}
t_k=\left(1-\frac{\beta}{5}\right)^2(1+\theta^2+\ldots +\theta^{2(k-1)})Er^2
\end{equation}
 starting from $k=1$, and we will be able to proceed with the iteration for as long as $\theta^kr<\frac{R}{4}$. Thus, at the last step we would have $\frac{R}{4\theta} \leq \theta^kr < \frac{R}{4}$  and as
\begin{equation}\label{time_scale_comp}
\frac{1}{2}\theta^{2(k-1)}Er^2 \leq t_k \leq \frac{\theta^{2k}Er^2}{16}  
\end{equation}
 (since $\theta>20$) this gives existence for time duration $c_0R^2$  with
\begin{equation}
c_0=\frac{E}{32\theta^4}.
\end{equation} 
Using the second inequality of \eqref{time_scale_comp}, the estimates above imply that
\begin{equation}
|A(t_k)| \leq \frac{2\alpha}{\theta^{k}r} \leq \frac{\alpha\sqrt{E}}{\sqrt{t_k}}. 
\end{equation}
Similarly, letting $(\phi_t)_{t\in [0,c_0R^2]}$ be the flow starting from $Y$, using the first inequality of \eqref{time_scale_comp} we have
\begin{equation}\label{mot_par}
d(\phi_{t_k}(y),\phi_0(y))\leq \frac{\beta}{2} \theta^{k}r \leq \frac{\beta\theta\sqrt{t_k}}{\sqrt{E}} 
\end{equation}
and in particular
\begin{equation}
d_H(Y_{t_k},X)\leq \frac{\beta\theta\sqrt{t_k}}{\sqrt{E}}+\frac{\beta}{2}r \leq \frac{2\beta\theta\sqrt{t_k}}{\sqrt{E}}. 
\end{equation} 

\vspace{5 mm}
The above calculations are only valid for the iteration times, but choosing $\de$ sufficiently small, we can make $A(s),B(s),C(s)$ hold for all $Y_{t}$ with $t\in [0,Er^2)$ and $r \leq s \leq \theta r$. Thus, we get that for every $t\in [2Er^2 , c_0R^2]$  we have
\begin{equation}
|A(t)| \leq \frac{c_1}{\sqrt{t}}
\end{equation}
and
\begin{equation}
d_H(Y_{t},X)\leq c_2\sqrt{t}. 
\end{equation} 
with $c_1=2\alpha\sqrt{E}$ and $c_2=\frac{2\beta\theta}{\sqrt{E}}$. Note that by Lemma \ref{main_lemma}, $c_1^2<\frac{1}{80}$ and $c_1c_2<\min(10^{-6},\Lambda^{-1})$.

\vspace{5 mm}

\noindent For the last part of Theorem \ref{main_thm2}, fix $c_3=4\alpha^2E$, let $t\in [c_3r^2,c_0R^2]$ and $s\in(\frac{\sqrt{t}}{2\alpha\sqrt{E}},\frac{R}{4})$.  Since $s\geq\frac{\sqrt{t}}{2\alpha\sqrt{E}}\geq r$, by the third condition of Theorem \ref{aux_reif_thm}, for every $x\in X$ one can write
\begin{equation}
B(x,s)\cap X^r_0=G\cup B
\end{equation}
where $G$ is connected and $B\cap B(x,(1-20\de)s)= \emptyset$, and as $\de<\frac{\beta}{10}<\frac{E}{10}<\frac{\theta E}{10}$, $B\cap B(x,(1-2\theta E)s)=\emptyset$. On the other hand, note that by \eqref{mot_par}, we have the motion bound
\begin{equation}
d(\phi_{t}(y),\phi_0(y))\leq  \frac{2\theta\beta\sqrt{t}}{\sqrt{E}}=\frac{\theta E\sqrt{t}}{\alpha\sqrt{E}}\leq 2\theta E s
\end{equation}
so we obtain that 
\begin{equation}
B(x,s)\cap X^r_t=G_t\cup B_t
\end{equation}
where $G_t$ is connected and $B_t\cap B(x,(1-4\theta E)s)= \emptyset$. In particular $B_t\cap B(x,\frac{9}{10}s)= \emptyset$ 
$\Box$.
 
\vspace{5 mm}

Having Theorem \ref{main_thm2} established, the existence parts of Theorem \ref{real_main_thm} follow easily. 
\begin{theorem}\label{smooth_ws}
With the $\de,c_0,c_1,c_2,c_3$ of Theorem \ref{main_thm2}, if $X$ is an $(\de,R)$-Reifenberg set, there exists a smooth MCF $(X_t)_{t\in(0,c_0R^2)}$ satisfying
\begin{equation}
\lim_{t\rightarrow 0}d_H(X_t,X)=0.
\end{equation}
Moreover, setting $X_0=X$, this solution is also a weak set flow (see Definition \ref{ls_flow}) and a physical solution (see Definition \ref{app_phys_sol}).  
\end{theorem}
\proof Using the estimates of Theorem \ref{main_thm2} and the standard estimates for the higher derivatives of $A$, the existence of a smooth solution $(X_t)_{t\in (0,c_0R^2]}$ follows immediately by an Arzel\`{a}-Ascoli argument. Since the flow $(X_t)_{t\in (0,c_0R^2)}$ is smooth, we only need to check avoidance with respect to smooth flows starting at time $0$. Take a smooth flow $(\Delta_t)_{t\in [0,T]}$ for $T\leq c_0R^2$ with 
\begin{equation}
d_H(\Delta_0,X)=\delta>0.
\end{equation} 
For $r$ sufficiently small we have 
\begin{equation}
d_H(X^r,\Delta_0)\geq \frac{\delta}{2}
\end{equation}
so by avoidance we will also have 
\begin{equation}
d_H(X^r_t,\Delta_t)\geq \frac{\delta}{2}.
\end{equation}
As $X_t=\lim_{j\rightarrow \infty}X^{r_j}_t$ with $r_j \rightarrow 0$ as $j\rightarrow \infty$ this implies 
\begin{equation}
X_t\cap \Delta_t=\emptyset
\end{equation} 
$\Box$.

\section{Uniqueness}\label{uniq_sec}
In this section we will  conclude the proof of Theorem \ref{real_main_thm}. In Section \ref{graph_rep_sec} we will see that at any positive time, the approximating flows of small enough scales can be written as graphs over one another. In Section \ref{graph_pde_sec} we will derive a PDE for such a situation which will be used to obtain the separation estimate, Theorem \ref{main_thm3} in Section \ref{uniq_lim_sec}. Finally, we will conclude the proof of Theorem \ref{real_main_thm}  in Section \ref{wf_sec} (assuming the interior estimate, Theorem  \ref{main_estimate}, which will in turn be proved in Section \ref{est_sec}). Throughout this section we will assume freely that the parameters $\alpha,\beta,\theta,c_0$ satisfy the inequalities  of Theorem \ref{main_thm2}.  

\subsection{Graph Representation}\label{graph_rep_sec} 

\noindent Recall that $\left(Y^r_t\right)_{t\in(0,c_0R^2)}$ is called a ($\Lambda$) $r$-approximate evolution of $X$ if for $t\in [c_3r^2,c_0R^2]$
\begin{equation}\label{app_1}
|A^r(t)| \leq \frac{c_1}{\sqrt{t}},
\end{equation}
\begin{equation}\label{app_2}
d_H(Y^{r}_t,X) \leq c_2\sqrt{t},
\end{equation}
and for any $t\in [c_3r^2,c_0R^2]$ and $s\in (\frac{\sqrt{t}}{c_1},R/4)$ and $x\in X$ we have:
\begin{equation}\label{app_3}
B(x,s)\cap Y^r_t=G\cup B
\end{equation}
where $G$ is connected and $B\cap B(x,\frac{9}{10}s)=\emptyset$. 
We will always assume that the parameters satisfy the relations of Theorem \ref{main_thm2}. To be more precise, we assume:
\begin{enumerate}\label{struct_ineq}
\item $c_1c_2<\min\{10^{-6},\Lambda^{-1}\}$.
\item $c_1^2<\frac{1}{80}$.
\end{enumerate}

\vspace{5 mm}

\noindent We start by making the following two observations.

\begin{lemma}\label{taub_lemma}
If $\left(Y^r_t\right)_{t\in (0,c_0R^2)}$ is an $r$-approximate evolution of $X$ then for every $t\in[c_3r^2,c_0R^2]$
\begin{equation}
\mathrm{inj}(Y^r_t)\geq \min\left(\frac{\sqrt{t}}{4c_1},\frac{1}{4}R\right).
\end{equation}
 Here $\mathrm{inj}$ denotes the injectivity radius of the normal exponential map. In particular, $Y^{s}_t$ is contained in a tubular neighborhood of $Y^{r}_t$  for every $s\leq r$.
\end{lemma}
\proof Set $\rho$ to be the distance function from $Y^r_t$. The curvature bound \eqref{app_1} imply that there are no focal points with $\rho \leq \frac{\sqrt{t}}{c_1}$. By the characterization of the injectivity radius we know that if $p$ is a cut point with $\rho(p)=\mathrm{inj}(Y^r_t)\leq\frac{\sqrt{t}}{4c_1}$ then there exist $y_1,y_2\in Y^r_t$ such that 
\begin{equation}\label{close_points_1}
d(y_1,y_2) \leq \frac{\sqrt{t}}{2c_1}
\end{equation}
and such that 
\begin{equation}\label{par_tan}
(y_1-y_2)\perp T_{y_i}Y^{r}_t
\end{equation}
for $i=1,2$. Fix $x\in X$ such that 
\begin{equation}
d(x,y_1)\leq c_2\sqrt{t}=\frac{c_1c_2\sqrt{t}}{c_1} <10^{-6}\frac{\sqrt{t}}{c_1} 
\end{equation}
and consider the intersection $B(x,\frac{\sqrt{t}}{c_1})\cap Y^r_t=G\cup B$ where the splitting is by \eqref{app_3}.  By \eqref{close_points_1} we know that $y_1,y_2\in G$ and since $y_1$ is very close to $x$ the curvature bound \eqref{app_1} together with the perpendicularity of the tangent spaces \eqref{par_tan} will force the connected component of $y_1$ to leave the ball before being able to return to $y_2$, which is a contradiction $\Box$.

\vspace{5 mm}

\noindent In fact, we even have the following.

\begin{lemma}
For $t\in [c_3r^2,c_0R^2]$ and $s\leq r$, $Y^{s}_t$ is a  a graph of a function $u$ over $Y^{r}_t$. By that we mean that
\begin{equation}
Y^{s}_t=\{y+u(y,t)N(y,t)\;|\;y\in Y^r_t\}.
\end{equation}
where $N(y,t)$ is the normal to $Y^r_t$ at $y$.
\end{lemma}
\proof By the above lemma, we know that $Y^{s}_t \subseteq \mathcal{N}(Y^{r}_t)=\mathcal{N}=B(Y^r_t,\mathrm{inj}(Y^r_t))$. Denote $Y=Y^r_t$ and $Z=Y^{s}_t$ and denote by $\pi:\mathcal{N}\rightarrow Y$ the projection to the nearest point. 

\noindent On one direction, suppose that there are $z_1\neq z_2 \in Z$ such that $\pi(z_1)=\pi(z_2)=y\in Y$. By applying \eqref{app_2} for both  $Y$ and $Z$, we see that 
\begin{equation}
d(z_i,y) \leq 2\cdot c_2\sqrt{t} \leq 2\cdot10^{-6} \frac{\sqrt{t}}{c_1}
\end{equation}
and that there exists $x\in X$ such that 
\begin{equation}
d(x,y) \leq 10^{-6} \frac{\sqrt{t}}{c_1}
\end{equation}
By \ref{app_3} we can write
\begin{equation}
B\left(x,\frac{\sqrt{t}}{c_1}\right)\cap Z=G\cup B 
\end{equation}
and $z_1,z_2\in G$. The curvature bounds and distance bounds \eqref{app_1},\eqref{app_2} will not allow that. In fact, by the following lemma, we have that
\begin{equation}
\tan(\angle(T_yY,T_{z_i}Z)) \leq \frac{6}{1000}
\end{equation}
As $(z_i-y) \perp T_yY$ we see that $\frac{z_2-z_1}{||z_2-z_1||}$ is almost perpendicular to $T_{z_i}Z$ and so $z_1$ and $z_2$ are two points in $G$ that are very close to the center of $B\left(x,\frac{\sqrt{t}}{c_1}\right)$ with almost parallel tangent planes that lie one above the other. The curvature bound \eqref{app_1} prevents that since, as before, it will force the connected component of $z_1$ to leave the ball before returning to $z_2$.  Thus over every point in $Y$ lies at most one point in $Z$. 

On the other hand, given any $y\in Y$  there is $z'\in Z$ with
\begin{equation}
d(z',y) \leq 2\cdot 10^{-6}\frac{\sqrt{t}}{c_1}
\end{equation}
Letting $y'=\pi(z')$ we see again that $\tan(\angle(T_{y'}Y,T_{z'}Z)) \leq \frac{6}{1000}$, so by the curvature bounds \eqref{app_1}, and since $y$ and $y'$ are very close compared to the scale $\frac{\sqrt{t}}{c_1}$, there will be a point $z$ over $y$.
This completes the proof $\Box$.

\begin{lemma}\label{interpolation2}
Let $M^2\subseteq\mathcal{N}(M^1)$ where $M^1,M^2$ are hypersurfaces in $\mathbb{R}^{n+1}$ and let $\pi:\mathcal{N}(M^1)\rightarrow M^1$ be the nearest point projection. If we have the bounds $\max\left(|A^1|,|A^2|\right)\leq a$  and $d_H(M^1,M^2) \leq b$ and $|ab|<<1$ then for every $m_2\in M_2$ setting $m_1=\pi(m_2)$ we have
\begin{equation}
\tan(\angle(T_{m_1}M^1,T_{m_2}M^2)) \leq \frac{(2-ab)\sqrt{ab-\frac{a^2b^2}{4}}}{2\left(1-\frac{ab}{2}\right)^2-1} \leq 3\sqrt{ab}.
\end{equation}
\end{lemma}
\proof The proof is similar to the one of  Lemma \ref{interpolation}. This time we obtain a lower bound for $|u|$. Assume the tangent spaces are not the same (else there is nothing to prove) and that w.l.g that $T_{m_1}M^1=\spa\{e_1,\ldots e_n\}$ $m_1=0$ and $m_2=(0,\ldots ,0,z)$ with $z\geq 0$.  By drawing two balls of radius $\frac{1}{a}$, one starting horizontally that bounds $M^1$ from above and the other starting parallel to $T_{m_2}M^2$ that bounds $M^2$ from below, those two barriers will force $M^2$ to drift apart from $M^1$ until the angle will be halved. This will give a lower bound on the Hausdorff distance, which should be less than $b$ by assumption. The details are left to the reader $\Box$.  

\subsection{The Graph PDE for Two Evolving Surfaces}\label{graph_pde_sec}

Having discovered that $Y^{s}_t$ is a graph of a function $u$ over $Y^r_t$ when $s\leq r$ and $t\geq c_3r^2$, let us compute the evolution equation in such a situation. The computation done here is the parabolic analogue for the one done in \cite[Lemma 2.26]{CM} for the minimal surface case.

\begin{lemma}\label{graph_pde}
Let $(M^1_t)_{t\in (0,T)}$ flow by mean curvature and let $\mathcal{N}_t$ be a tubular neighborhood of $M^1_t$. Let $(M^2_t)_{t\in (0,T)}$ be another flow by mean curvature such that for every $t\in (0,T)$ $M^2_t \subseteq \mathcal{N}_t$. Let $\vec{N}(x,t)$ denote the inner pointing unit normal to $M^1_t$ at $x$ and write $M^2_t$ as a graph of a function $u(x,t)$ over $M^1_t$, so
\begin{equation}
M^2_t=\{x+u(x,t)N(x,t)\;|\;x\in M^1_t\}.
\end{equation}
Then $u$ satisfies an equation of the form:
\begin{equation}\label{EL}
u_t=(1+\de)\mathrm{div}((I+L)\nabla u)+|A|^2 u - Q_{ij}A_{ij}
\end{equation}
where $A$ is the second fundamental form of $M^1$, $A_{ij}$ are its components in a local orthonormal frame, $|\de|,|L|\leq D(|A||u|+|\nabla u|)$ and $|Q_{ij}| \leq D(|A||u|+|\nabla u|)^2$ for some constant $D>0$. 
\end{lemma}
\proof  We want to derive an expression for the mean curvature of $M^2$   $\vec{H}_2$, for some fixed time $t$, in terms of $u$ and its derivatives, and in terms of the geometric quantities of $M^1$. Once we have done that, the derivation of the equation will follow easily. Let $A$,$H$, $dVol$ be the second fundamental form, the mean curvature and the volume form of $M^1$. Choose $m_1\in M^1$ and let $E_1,\ldots,E_n$ be a local orthonormal basis for $M^1$ around $m_1$. Extend $E_1,\ldots,E_n$, along with $N$, to normal fields in a neighborhood of $m_1$ in $\mathcal{N}$ by parallel translating along $N$. The vectors
\begin{equation}
X_i=E_i+u_iN-uA_{ik}E_k
\end{equation} 
form a basis to the tangent space of $M^2$. Therefore
\begin{equation}
g_{ij}=\delta_{ij}+u_iu_j-2uA_{ij}+u^2A_{ik}A_{jk}=\delta_{ij}-2uA_{ij}+Q_{ij}
\end{equation}
and so 
\begin{equation}
g^{ij}=\delta_{ij}+2uA_{ij}+Q_{ij}.
\end{equation}
Now, taking a variation $\tilde{u}(x,s)=u(x)+sv(x)$ for some function $v:M^1 \rightarrow \mathbb{R}$ which is localized in the above neighborhood and computing the corresponding quantities, we get
\begin{equation}
\begin{aligned}
\frac{d}{ds}|_{s=0}\sqrt{\det g_{ij}}=&\frac{1}{2}\sqrt{\det g_{ij}}\mathrm{trace}(g^{ij}\frac{d}{ds}|_{s=0}g_{ij})= \\
&\frac{1}{2}\sqrt{\det g_{ij}}(\delta_{ij}+2uA_{ij}+Q_{ij})(v_iu_j+u_iv_j-2vA_{ij}+2uvA_{ik}A_{jk}) = \\
&\sqrt{\det g_{ij}}\left[\left<\nabla u,\nabla v\right>-vH-uv|A|^2+\left<L\nabla u,\nabla v\right>+Q_{ij}A_{ij}v\right]= \\
& \left<(I+L)\nabla u,\nabla v\right> + \sqrt{\det g_{ij}}\left[-vH-uv|A|^2+Q_{ij}A_{ij}v\right] 
\end{aligned}
\end{equation}
where in the last equality we have used
\begin{equation}\label{det_ord}
\sqrt{\det g_{ij}}=1-uH+Q
\end{equation}
to absorb $\sqrt{\det g_{ij}}$ into $(I+L)$. Thus, by integration by parts we get
\begin{equation}
-\frac{d}{ds}|_{s=0}\int \sqrt{\det g_{ij}}dVol=\int \sqrt{\det g_{ij}}v\left[(1+\de)\mathrm{div}((I+L)\nabla u)+H+u|A|^2+Q_{ij}A_{ij}\right]dVol
\end{equation}
where we have used \eqref{det_ord} to write $\frac{1}{\sqrt{\det{g_{ij}}}}=1+\de$. On the other hand, by the definition of the mean curvature vector, we obtain
\begin{equation}
-\frac{d}{ds}|_{s=0}\int \sqrt{\det g_{ij}}dVol=\int v\left<N,\vec{H}_2\right>\sqrt{\det g_{ij}}dVol
\end{equation}
so
\begin{equation}
\left<N,\vec{H}_2\right>=(1+\de)\mathrm{div}((I+L)\nabla u)+H+u|A|^2+Q_{ij}A_{ij}.
\end{equation}
Having done that, deriving the PDE is easy. Indeed, writing $M^i$ as $\phi^i(x,t)$ we get
\begin{equation}
\phi^2(x,t)=\phi^1(x,t)+u(x,t)N(x,t)
\end{equation}
so differentiating with respect to time we obtain
\begin{equation}
\vec{H}_2=\vec{H}+u_t\vec{N}-u\nabla H
\end{equation}
so
\begin{equation}
u_t=\left<\vec{H}_2,N\right>-H=(1+\de)\mathrm{div}((I+L)\nabla u)+u|A|^2+Q_{ij}A_{ij}
\end{equation}
$\Box$.

\subsection{The Separation Estimate}\label{uniq_lim_sec}

We now come to the proof of the separation estimate, Theorem \ref{main_thm3}. In fact, we are going to prove a slightly stronger theorem, from which Theorem \ref{main_thm3} will follow easily. At long last, we are going to fix the parameter $\Lambda$, on which the output of Theorem \ref{main_thm2} and thus the definition of an $r$-approximate evolution depended. 

\begin{theorem}\label{improved_est}
Let $s\leq r$ and let $Y^r_t$ and $Z^s_t$ be $(\Lambda)$ $r$-approximate and $s$-approximate evolutions of $X$ with $\Lambda=2D$, where $D$ is the constant from Lemma \ref{graph_pde}.  Writing $Y^{s}_t$ as a graph of a function $u$ over $Y^r_t$ for $s\leq r$ and $t\in [c_3r^2,c_0R^2]$, we have the estimate
\begin{equation}
|u| \leq Cr^{1/2}t^{1/4}
\end{equation}
for some $C>0$.
\end{theorem}  
\proof The idea is to use equation \eqref{EL} and the maximum principle to bootstrap and obtain more and more improved estimates for $u$. In order to start the bootstrap, note that by \eqref{app_1} and \eqref{app_2} we have
\begin{equation}\label{iter_curv_est}
|A(t)|=|A^r(t)| \leq \frac{c_1}{\sqrt{t}}
\end{equation}
and
\begin{equation}\label{iter_u_est}
|u(t)|  \leq 2c_2\sqrt{t}.
\end{equation}
Notice that at a maximum point of $u$, the $|\nabla u|$ dependence of the bounds on the coefficients in equation \eqref{EL} disappears. Note further that by Cauchy-Schwartz and estimates \eqref{iter_curv_est},\eqref{iter_u_est}
\begin{equation}
|Q_{ij}A_{ij}| \leq D(|A||u|)^2|A| \leq 2Dc_1c_2 |u||A|^2\leq |u||A|^2
\end{equation}
where we have used the relation $c_1c_2<\Lambda^{-1}=(2D)^{-1}$.  Therefore, using estimates \eqref{iter_curv_est},\eqref{iter_u_est} and applying the maximum principle for equation \eqref{EL} we obtain
\begin{equation}
\frac{d}{dt}u(x_{\max},t) \leq 2u(x_{\max},t)|A(x_{\max},t)|^2 \leq 2c_2\frac{2c_1^2}{\sqrt{t}} < \frac{2c_2}{40}\frac{1}{\sqrt{t}}
\end{equation}
where the last inequality holds since $c_2^2<\frac{1}{80}$. A similar calculation is done for $u(x_{\min})$. By integrating starting from $t_0=c_3r^2$ we get the improved estimate
\begin{equation}
|u(t)| \leq 2c_2\left(\sqrt{c_3} r+\frac{1}{10}\sqrt{t}\right)\leq  2c_2\left(c_3^{1/4} r^{1/2}t^{1/4}+\frac{1}{10}\sqrt{t}\right). 
\end{equation}
Plugging this back into the inequality, we get
\begin{equation}
\frac{d}{dt}u(x_{\max},t) \leq 2c_2\left(\frac{1}{40}\frac{c_3^{1/4}r^{1/2}}{t^{3/4}}+\frac{1}{400}\frac{1}{\sqrt{t}}\right)
\end{equation}
which yields
\begin{equation}
|u(t)| \leq 2c_2\left(\left(1+\frac{1}{10}\right)c_3^{1/4} r^{1/2}t^{1/4}+\frac{1}{100}\sqrt{t}\right).
\end{equation}
Continuing the bootstrap, this will yield:
\begin{equation}
|u(t)| \leq 4c_2c_3^{1/4} r^{1/2}t^{1/4}
\end{equation}
as required $\Box$.
\proof[Proof of Theorem \ref{main_thm3}] This follows directly from Theorem \ref{improved_est} $\Box$.
\proof[Proof of Corollary \ref{weak_uniq_cor}] Fixing $t>0$, for $r$ sufficiently small we have $t \geq c_3r^2$. By Theorem \ref{main_thm3} for every $s\leq r$
\begin{equation}
d_H\left(Y^{r}_t,Y^{s}_t\right)\leq C(t)\sqrt{r}. 
\end{equation} 
Thus, for every sequence $\{r_n\}_{n=1}^\infty$ with $r_n>0$ and $\lim_{n\rightarrow \infty}r_n=0$, $\{Y^{r_n}_t\}_{n=1}^\infty$ forms a Cauchy sequence in the Hausdorff sense. Therefore $\lim_{r\rightarrow 0}Y^r_t$ exists. The second part follows since every physical solution is an $r$-approximate evolution for every $r>0$, so given two physical flows, we can use theorem \ref{main_thm3} with any choice of $r$ $\Box$. 
\vspace{5 mm}

\subsection{The Level-Set Flow}\label{wf_sec}
We start this section by adapting what we have done so far a little bit. According to Corollary  \ref{aux_reif_cor} we can choose approximations $X^r_+$ and $X^r_-$ instead of the $X^r$, such that if $\Omega$ is the bounded open domain with $\partial \Omega=X$ then $X^r_+ \subseteq \bar{\Omega}^c$ and $X^r_- \subseteq \Omega$. There are two differences  between the estimates for $X^r$ that we have been working with and the ones for $X^r_\pm$. The first one is that the constants $c_1,c_2$ in Theorem \ref{aux_reif_thm} and Corollary \ref{aux_reif_cor} are different. The second difference is in property (3), namely, the marginal connected components of $X^r\cap B(x,s)$ for $s\geq r$ are inside the annulus $(B(x,s)-B(x,(1-20\de) s))$, while for $X^r_\pm$ they will be in the slightly thicker annulus $(B(x,s)-B(x,(1-40\de) s))$. These two differences clearly do not influence our arguments in any way so in particular, Theorem \ref{main_thm2} is still valid for the flows emanating from $X^r_{\pm}$. Thus choosing $\Lambda$ as in Proposition \ref{improved_est} and with the $\de,c_0,\ldots c_3$ of Theorem \ref{main_thm2}, we know that  $X^r_{\pm,t}$ exist for time duration $c_0R^2$  and that
\begin{equation}
\lim_{r\rightarrow 0}X^r_{-,t}=\lim_{r\rightarrow 0}X^r_{+,t}=X_t
\end{equation} 
where $X_t$ is the unique physical flow emanating from $X$.
\proof[Proof of Theorem \ref{real_main_thm}] Let $\tilde{X}_t$ be the level set flow of $X$.  By Theorem \ref{smooth_ws} we know that $X_t \subseteq \tilde{X}_t$.  On the other hand, by the avoidance principle, we know that since $X\cap X^r_\pm=\emptyset$  we have 
\begin{equation}\label{no_inter}
\tilde{X}_t\cap X^r_{\pm,t}=\emptyset.
\end{equation}
Denoting by $\Omega^r_{\pm,t}$ the bounded domains with $\partial \Omega^r_{\pm,t}=X^r_{\pm,t}$  we claim that 
\begin{equation}\label{where_lives}
\tilde{X}_t \subseteq \bar{\Omega}^r_{+,t}-\Omega^r_{-,t}:=N^r_t.
\end{equation}
This is true since considering
\begin{equation}
T^r=\inf\{0 \leq t \leq c_0R^2\;\;|\;\; \tilde{X}_t\cap (N_t^r)^c \neq \emptyset\}
\end{equation}
we clearly have that $T^r>0$ by avoiding balls, since the two sets start a positive distance apart. Given $(x_i,t_i)$ such that $x_i \in \tilde{X}_{t_i}\cap (N_{t_i}^r)^c$ with $t_i \rightarrow T^r$ we see (by avoiding balls and since the evolution of $\partial N^r_t=X^r_{+,t}\cup X^r_{-,t}$ is smooth) that $d(x_i,\partial N^r_{t_i})\rightarrow 0$. By compactness  we get $\lim_{j\rightarrow \infty}x_{i_j}=x\in \partial N^r_{T^r}$. On the other hand, by the level set definition of the level set flow, $\tilde{X}_t=f_t^{-1}(0)$ for some continuous function $f:[0,c_0R^2]\times \mathbb{R}^{n+1}\rightarrow \mathbb{R}$ (see \cite{ES1},\cite{CGG}) and so we see that  $x\in \tilde{X}_{T^r}$. Thus $x\in \tilde{X}_{T^r}\cap X^r_{\pm,T^r}$ which contradicts \eqref{no_inter}. 

\vspace{5 mm}
Having \eqref{where_lives} established and noting that by Theorem \ref{main_thm3},  $d_H(N^r_t,X^r_{+,t})\leq C(t)r^{1/2}$, we finally conclude that
\begin{equation}
\tilde{X}_t \subseteq  \bigcap_{0<r<\frac{R}{4}} N^r_t \subseteq \lim_{r\rightarrow 0}N^r_t=X_t.
\end{equation}
Thus $\tilde{X}_t=X_t$ and we are done $\Box$.

\section{Proof of Theorem \ref{main_estimate}}\label{est_sec}

The main purpose of this section is to prove  Theorem \ref{main_estimate}, which was an essential ingredient in the proof of Theorem \ref{main_thm2} and is thus a corner stone for the entire argument. In Section \ref{me_sec} we will state Theorem \ref{main_thm2} in slightly different terms. In Section \ref{sch_sec} we will derive an a priori estimate for the non-homogeneous heat equation which is suitable for thick space-time cylinders. In Section \ref{hold_sec} we will derive a H\"{o}lder gradient estimate suitable for our situation and in Section \ref{conc_sec} we will conclude the proof of Theorem \ref{main_est}.  

\subsection{Main Estimates}\label{me_sec}

 Theorem \ref{main_thm2} will follow immediately from the following ``centered form'' of the estimate,  the proof of which will be described through the rest of this section.

\begin{theorem}[main estimate II]\label{main_est}
There exist some $c>0$ such that for every $\delta>0$ and $M>0$, there exist positive $\tau_0=\tau_0(M,\delta)<<1$ and $\lambda_0=\lambda_0(M,\delta)<<1$ such that for every $0<\lambda<\lambda_0$ there exists  some  $\de_0=\de_0(\delta,M,\lambda)$ such that for every $0<\tau<\tau_0$  and $\de<\de_0$ the following holds:

\noindent If $u:B(p,r)\times [0,\tau r^2]\rightarrow \mathbb{R}$ is a graph moving by mean curvature such that:
\begin{enumerate}
\item For every $(x,t) \in B(p,r)\times [0,\tau r^2]$ 
\begin{equation}
|\nabla u(x,t)| \leq M\de. 
\end{equation}
\item  For every $x \in B(p,r)$ we have
\begin{equation}
|\nabla u(x,\lambda \tau r^2)| \leq \de.
\end{equation}
\item For every $(x,t) \in B(p,r)\times [0,\tau r^2]$ 
\begin{equation}
|u(x,t)| \leq M^2\beta. 
\end{equation}
\item  For every $x \in B(p,r)$ we have
\begin{equation}
|u(x,\lambda \tau r^2)| \leq \beta.
\end{equation}
\end{enumerate}Then:
\begin{enumerate}[label=\Alph*.]
\item 
\begin{equation} 
|A(p,\tau r^2)| \leq (1+\delta)\frac{1}{\sqrt{\pi}}\frac{\de}{\sqrt{\tau} r}.
\end{equation}
\item 
\begin{equation}
|A(p,\tau r^2)| \leq c\frac{\beta}{\tau r^2}+\delta\frac{\de}{\sqrt{\tau} r}
\end{equation}
\end{enumerate}
\end{theorem}

\subsection{Estimates for the Non-Homogeneous Heat Equation on Thick Cylinders}\label{sch_sec}

The first key point in obtaining the main estimate - Theorem \ref{main_est} -  is a non standard Schauder-type estimate for the non homogeneous heat equation.  Before stating it, we record the following definitions.

\begin{definition}
The parabolic ball of radius $r$ with center $(p,t)$  is defined to be
\begin{equation}
P(p,t,r)=B(p,r)\times [t-r^2,t]
\end{equation}
when it is clear from the context which point is $p$, we define $P^r=P(p,r^2,r)$. For $\tau,\lambda<1$  we  define the narrow $r,\tau$ cylinder to be
\begin{equation}
P^{r,\tau}=B(p,r)\times[0,\tau r^2]
\end{equation}
and the $\lambda$ truncated narrow $r,\tau$ cylinder to be
\begin{equation}
P^{r,\tau,\lambda}=B(p,(1-\sqrt{\lambda}) r)\times [\lambda  \tau r^2, \tau r^2].
\end{equation}
\end{definition}

\begin{theorem}\label{Schauder}
There exists a constant $c>0$ such that for every $\delta>0$ and $M>0$ there exist positive $\tau_0=\tau_0(M,\delta)<<1$ and $\lambda_0=\lambda_0(M,\delta)<<1$  such that for every $0<\lambda<\lambda_0$ and $0<\alpha<1$ there is a constant $C=C(\lambda,\alpha)>0$  such that for every $0<\tau<\tau_0$ the following holds:

\noindent If $u$ is a solution to the non-homogeneous heat equation
\begin{equation}
u_t-\Delta u=f
\end{equation}
on $B(p,r)\times [0,\tau r^2]$ such that:
\begin{enumerate}
\item For every $(x,t) \in B(p,r)\times [0,\tau r^2]$ 
\begin{equation}
|\nabla u(x,t)| \leq M\de. 
\end{equation}
\item  For every $x \in B(p,r)$ we have
\begin{equation}
|\nabla u(x,\lambda \tau r^2)| \leq \de.
\end{equation}
\item For every $(x,t) \in B(p,r)\times [0,\tau r^2]$ 
\begin{equation}
|u(x,t)| \leq M^2\beta. 
\end{equation}
\item  For every $x \in B(p,r)$ we have
\begin{equation}
|u(x,\lambda \tau r^2)| \leq \beta.
\end{equation}
\end{enumerate}
Then:
\begin{enumerate}[label=\Alph*.]
\item 
\begin{equation}
\begin{aligned}
&\sqrt{\tau}r|\nabla^2 v(0,\tau r^2)| \leq \\
& (1+\delta)\frac{1}{\sqrt{\pi}}\de + \frac{C(\alpha,\lambda)}{\sqrt{\tau}r}\left(\sup_{z_1\in P^{r,\tau}} d^2_{z_1}|f(z_1)|+\sup_{z_1,z_2\in  P^{r,\tau}} d^{2+\alpha}_{z_1,z_2}\frac{|f(z_2)-f(z_1)|}{d(z_2,z_1)^{\alpha}}\right).
\end{aligned}
\end{equation}
\item 
\begin{equation}
\begin{aligned}
&\sqrt{\tau}r|\nabla^2 v(0,\tau r^2)| \leq \\ 
&c\frac{\beta}{\sqrt{\tau}r} + \frac{C(\alpha,\lambda)}{\sqrt{\tau}r}\left(\sup_{z_1\in P^{r,\tau}} d^2_{z_1}|f(z_1)|+\sup_{z_1,z_2\in  P^{r,\tau}} d^{2+\alpha}_{z_1,z_2}\frac{|f(z_2)-f(z_1)|}{d(z_2,z_1)^{\alpha}}\right).
\end{aligned}
\end{equation}
\end{enumerate}
Here 
\begin{equation}
d((x_1,t_1),(x_2,t_2))=\sqrt{|x_1-x_2|^2+|t_1-t_2|}.
\end{equation}
and  $d_{z_1}=d(z_1,\partial D)$, $d_{z_1,z_2}=\min(d_{z_1},d_{z_2})$.
\end{theorem}

\begin{remark}
There are two main differences between the standard Schauder estimate and the one above. The first and less important one is that
in the standard Schauder estimates the term $|u|_0$ appears in the right hand side instead of $|\nabla u|_0$. In the elliptic case, an estimate similar to the one above follows from the standard Schauder estimate by integration. In the parabolic case this is not as trivial, as we do not assume anything about $|u_t|$. The second and more important difference is that in our estimate there is a distinction between a leading term coming from the initial time slice and a negligible term coming from the rest of the parabolic boundary. In the standard Schauder estimate there is no such distinction.
\end{remark}

In order to prove this version of  Schauder estimate, we need the following two lemmas. The first one is a standard result for Gaussian potentials (see for instance the proof of Theorem 2 in \cite[Chapter 4.3]{Fri} and Theorems 4.6 and 4.8 in \cite{GT}).
\begin{lemma}\label{potential_lemma}
For every $0<\alpha<1$ there exists a constant $C$ with the following property. Let $\Phi$ be the fundamental solution for the heat equation and let $w$ be the Gaussian potential corresponding to $f$, i.e
\begin{equation}
w(x,t)=\int_{0}^{t}\int_{B(0,r)} \Phi(x-y,t-s)f(y,s)dyds.
\end{equation}
Then $w_t-\Delta w=f$ and for every $z\in D \subseteq B(0,r)\times [\tau r^2]$ we have
\begin{equation} 
|w(z)|+d_z|\nabla w(z)|+d^2_z|\nabla^2 w| \leq C\left(\sup_{z_1\in D} d^2_{z_1}f(z_1)+\sup_{z_1,z_2\in D} d^{2+\alpha}_{z_1,z_2}\frac{|f(z_2)-f(z_1)|}{d(z_2,z_1)^{\alpha}}\right)
\end{equation}
$\Box$.
\end{lemma}

The following lemma gives a good interior derivative estimate for the heat equation at times that are very close to $0$ compared to the initial scale.

\begin{lemma}\label{harmonic_lemma}

For every $\delta>0$ and $M>0$  there exist a positive $\tau_0=\tau_0(M,\delta)<<1$  such that for every $0<\tau<\tau_0$ the following holds:

\noindent If $u$ is a solution to the homogeneous heat equation
\begin{equation}
u_t-\Delta u=0
\end{equation}
on $B(p,r)\times [0,\tau r^2]$ such that:
\begin{enumerate}
\item For every $(x,t) \in B(p,r)\times [0,\tau r^2]$
\begin{equation}
|u(x,t)| \leq M\de. 
\end{equation}
\item  For every $x \in B(p,r)$
\begin{equation}
|u(x,0)| \leq \de.
\end{equation}
\end{enumerate}
Then:
\begin{equation}
|\nabla u(p,\tau r^2)| \leq (1+\delta)\frac{1}{\sqrt{\pi}}\frac{\de}{\sqrt{\tau} r}.
\end{equation}
\end{lemma}

\proof The proof is modeled on the one of Theorem 2.3.8 in \cite{Eva}.

\noindent By scaling, it suffices to prove that if $u$ is such a solution on $B(0,R)\times [0,1]$ for $R$ sufficiently large we have
\begin{equation}
|\nabla u(0,1)| \leq (1+\delta)\frac{\de}{\sqrt{\pi}}. 
\end{equation}
First, fix a positive cutoff function $\xi \in C^\infty_0(B(0,R))$ with  $\xi|_{B(0,R-1)}=1$, $0 \leq \xi \leq 1$, $|\nabla \xi|,|\nabla^2 \xi| \leq C(n)$. Considering the function 
\begin{equation}
v(x,t)=\xi(x)u(x,t)
\end{equation}
we have that $v:\mathbb{R}^{n}\times [0,1]\rightarrow \mathbb{R}$ is defined and satisfies
\begin{equation}
v_t-\Delta v= -u\Delta \xi -2\nabla u \cdot \nabla \xi.
\end{equation}
As $v$ is bounded we have the representation formula
\begin{equation}
\begin{aligned}
v(x,t)= &\int_{\mathbb{R}^{n}}\Phi(x-y,t)v(y)dy+\\
             &\int_{0}^{t}\int_{\mathbb{R}^n}\Phi(x-y,t-s)\left(-u(y,s)\Delta \xi(y) -2\nabla u(y,s) \cdot \nabla \xi(y) \right)dyds = \\
             &\int_{\mathbb{R}^{n}}\Phi(x-y,t)v(y)dy+\\
             &\int_{0}^{t}\int_{\mathbb{R}^n}\left(\Phi(x-y,t-s)\Delta \xi(y) +2\nabla_y \Phi(x-y,t-s) \cdot \nabla \xi\right)u(y,s)dyds.
\end{aligned}
\end{equation}
Differentiating under the integral sign we obtain
\begin{equation}
\begin{aligned}
\nabla v(x,t)=& \int_{\mathbb{R}^{n}}\nabla_x\Phi(x-y,t)v(y)dy+ \\
& \int_{0}^{t}\int_{\mathbb{R}^n}\left(\nabla_x\Phi(x-y,t-s)\Delta \xi(y) +2\nabla_x\nabla_y \Phi(x-y,t-s) \cdot \nabla \xi\right)u(y,s)dyds.
\end{aligned}
\end{equation}
We will bound the two integrals appearing in this expression for $\nabla v(0,1)=\nabla u(0,1)$. To bound the first integral, we have, assuming w.l.g that $\nabla v(0,1) ||e_1$
\begin{equation}
|\nabla u(0,1)|=|\nabla v(0,1)|\leq ||v||_\infty \int \frac{e^{-|x|^2/4 }}{(4\pi)^{n/2}} \frac{x_i}{2}dx 
\end{equation}
so making the change of variables $(x_1, \ldots x_n) \rightarrow (\frac{x_1^2}{4},x_2,\ldots x_n)=(y,z)$
we get
\begin{equation}
|\nabla v(0,1)|\leq \frac{2}{\sqrt{4\pi}}||v||_\infty \int_{0}^{\infty} e^{-y}dy \int_{\mathbb{R}^{n-1}} \frac{e^{-|z|^2/4 }}{(4\pi)^{(n-1)/2}}dz  =\frac{1}{\sqrt{\pi}}||v||_{\infty}\leq \frac{1}{\sqrt{\pi}}\de. 
\end{equation}
As for the second integral, the integrand is zero outside the annulus $R-1<|y|<R$ so as both $\nabla \Phi$ and $\nabla^2 \Phi$ are summable on $(\mathbb{R}^{n}-B(0,1))\times (0,1)$ and by our assumptions on $u$ and $\xi$ we get that for $R$ large enough
\begin{equation}
|\nabla u(0,1)| \leq (1+\delta)\frac{1}{\sqrt{\pi}}\de
\end{equation}
$\Box$.

The following lemma is proved similarly.
\begin{lemma}\label{harmonic_lemma2}

The exist some $C>0$ such that for every $M>0$  there exist positive $\tau_0=\tau_0(M)<<1$  such that for every $0<\tau<\tau_0$ the following holds:

\noindent If $u$ is a solution to the homogeneous heat equation
\begin{equation}
u_t-\Delta u=0
\end{equation}
on $B(p,r)\times [0,\tau r^2]$ such that:
\begin{enumerate}
\item For every $(x,t) \in B(p,r)\times [0,\tau r^2]$
\begin{equation}
|u(x,t)| \leq M^2\beta. 
\end{equation}
\item  For every $x \in B(p,r)$
\begin{equation}
|u(x,0)| \leq \beta.
\end{equation}
\end{enumerate}
Then:
\begin{equation}
|\nabla u(p,\tau r^2)| \leq C\frac{\beta}{\tau r^2}
\end{equation}
$\Box$.
\end{lemma}

\proof[Proof of Theorem \ref{Schauder}] By scaling we can assume $r=1$. Let $v=u-w$ where $w$ is the Gaussian potential from Lemma \ref{potential_lemma}. Taking $\lambda<<1$ we therefore have the estimate
\begin{equation}
\sqrt{\lambda} \sqrt{\tau} |\nabla w|_{P^{1,\tau,\lambda}} \leq C(\alpha)\left(\sup_{z_1\in P^{1,\tau}} d^2_{z_1}|f(z_1)|+\sup_{z_1,z_2\in  P^{1,\tau}} d^{2+\alpha}_{z_1,z_2}\frac{|f(z_2)-f(z_1)|}{d(z_2,z_1)^{\alpha}}\right) 
\end{equation}
so
\begin{equation}
|\nabla w|_{P^{1,\tau,\lambda}} \leq \frac{C(\alpha,\lambda)}{\sqrt{\tau}}\left(\sup_{z_1\in P^{1,\tau}} d^2_{z_1}|f(z_1)|+\sup_{z_1,z_2\in  P^{1,\tau}} d^{2+\alpha}_{z_1,z_2}\frac{|f(z_2)-f(z_1)|}{d(z_2,z_1)^{\alpha}}\right). 
\end{equation}
Similarly
\begin{equation}\label{second_der_est}
\sqrt{\tau} |\nabla^2w|_{P^{1,\tau,\lambda}} \leq \frac{C(\alpha,\lambda)}{\sqrt{\tau}}\left(\sup_{z_1\in P^{1,\tau}} d^2_{z_1}|f(z_1)|+\sup_{z_1,z_2\in  P^{1,\tau}} d^{2+\alpha}_{z_1,z_2}\frac{|f(z_2)-f(z_1)|}{d(z_2,z_1)^{\alpha}}\right)
\end{equation}
and consequently:
\begin{enumerate}
\item for every $x\in B(p,(1-\sqrt{\lambda})r)$
\begin{equation}
|\nabla v(x,\lambda\tau)| \leq \de+\frac{C(\alpha,\lambda)}{\sqrt{\tau}}\left(\sup_{z_1\in P^{1,\tau}} d^2_{z_1}|f(z_1)|+\sup_{z_1,z_2\in  P^{1,\tau}} d^{2+\alpha}_{z_1,z_2}\frac{|f(z_2)-f(z_1)|}{d(z_2,z_1)^{\alpha}}\right).
\end{equation}
\item  for every $(x,t)\in P^{1,\tau,\lambda}$:
\begin{equation}
|\nabla v|_{P^{1,\tau,\beta,\lambda}} \leq M\de+M\frac{C(\alpha,\lambda)}{\sqrt{\tau}}\left(\sup_{z_1\in P^{1,\tau}} d^2_{z_1}|f(z_1)|+\sup_{z_1,z_2\in  P^{1,\tau}} d^{2+\alpha}_{z_1,z_2}\frac{|f(z_2)-f(z_1)|}{d(z_2,z_1)^{\alpha}}\right).
\end{equation}
\end{enumerate} 
Note that $v$ solves the heat equation and therefore so do $\frac{\partial v}{\partial x_i}$ for every $i=1,\ldots n$. By choosing $\tau_0$ and $\lambda$ sufficiently small, we obtain by Lemma \ref{harmonic_lemma}
\begin{equation}
\begin{aligned}
\sqrt{\tau}|\nabla^2 v(0,\tau)| \leq (1+\delta)\frac{1}{\sqrt{\pi}}\de + \frac{C(\alpha,\lambda)}{\sqrt{\tau}}\left(\sup_{z_1\in P^{1,\tau}} d^2_{z_1}|f(z_1)|+\sup_{z_1,z_2\in  P^{1,\tau}} d^{2+\alpha}_{z_1,z_2}\frac{|f(z_2)-f(z_1)|}{d(z_2,z_1)^{\alpha}}\right).
\end{aligned}
\end{equation}
By \ref{second_der_est} the desired result for $u$ follows. The second estimate is proved similarly, using Lemma \ref{harmonic_lemma2} $\Box$.

\subsection{H\"{o}lder Gradient Estimate}\label{hold_sec}

\noindent The second key ingredient in proving Theorem \ref{main_est} is the following H\"older gradient estimate for our equation.
\begin{theorem}\label{holder_grad}
There exist constants $c_1>0$ and $0<\alpha<1$, depending only on $n$,  such that if $\Omega \subseteq \mathbb{R}^n\times \mathbb{R}_+$ is a bounded domain and $u:\Omega\rightarrow \mathbb{R}$ solves the graphical mean curvature equation
\begin{equation}\label{graphical_MCF_PDE}
u_t=\left(\delta^{ij}-\frac{\partial_iu\partial_ju}{1+|\nabla u|^2}\right)\partial_i\partial_ju
\end{equation}
with $||\nabla u||_{\Omega,0}=\de<1$ then:
\begin{equation}
\sup_{z_1,z_2\in  \Omega} d^{\alpha}_{z_1,z_2}\frac{|\nabla u(z_1)-\nabla u(z_2)|}{d(z_2,z_1)^{\alpha}} \leq c_1\de
\end{equation}
\end{theorem}      
\proof This follows from tracing how the constants are formed in the interior Holder gradient estimate for quasilinear parabolic equations of general form, noticing that the derivatives of the coefficients, as well as the ellipticity of the equation resulted during the proof, are controlled. (see \cite[XII.3]{Lie} or \cite[13.3]{GT} in the elliptic case). As our situation is simpler than the general one, and for the convenience of the reader, we carry it out here, following \cite[XII.3]{Lie} and \cite[13.3]{GT}. 

\vspace{5 mm}

\noindent Equation \eqref{graphical_MCF_PDE} is of the form
\begin{equation}
u_t=a^{ij}(\nabla u)\partial_i\partial_ju
\end{equation}
with
\begin{equation}\label{graph_coeff}
a^{ij}(p)=\delta^{ij}-\frac{p_ip_j}{1+|p|^2}.
\end{equation} 
Differentiating with respect to $k$ and re-grouping we get
\begin{equation}
\partial_t(\partial_k u)=\partial_i\left(a^{ij}\partial_k\partial_ju\right)+a^{ijl}(\partial_k\partial_lu)(\partial_i\partial_ju)
\end{equation}
for $a^{ijl}=\partial_{p_l}a^{ij}-\partial_{p_j}a^{il}$. Considering the function $v=|\nabla u|^2$ we thus obtain
\begin{equation}
\partial_t v=\partial_i\left(a^{ij}\partial_jv\right)+a^{ijl}(\partial_lv)(\partial_i\partial_ju)-2a^{ij}(\partial_r\partial_ju)(\partial_r\partial_iu)
\end{equation}
so  fixing a parameter $\gamma$ and considering  
\begin{equation}
w=w^{\pm}=w^{\pm}_k= \pm \gamma\partial_k u + v
\end{equation}
we get:
\begin{equation}
-\partial_tw+\partial_i(a^{ij}\partial_jw)=-a^{ijl}(\partial_lw)(\partial_i\partial_ju)+2a^{ij}(\partial_r\partial_ju)(\partial_r\partial_iu).
\end{equation}
Since $|\nabla u| \leq 1$ , the explicit form of the coefficients $a^{ij}(p)$ given in \eqref{graph_coeff} implies that $\lambda |\xi|^2 \leq a^{ij}\xi_i\xi_j$, $|a^{ij}| \leq \Lambda$ and $a^{ijl} \leq c_2$ hold with $\lambda=\frac{1}{2},\Lambda=1$ and $c_2=6$. Thus, by ellipticity and by Schwarz's inequality we get, for some $c_3=c_3(n)$
\begin{equation}
-(\partial_tw)+\partial_i(a^{ij}\partial_jw) \geq -c_3|Dw|^2.
\end{equation} 
 Now, assume $(x_0,t_0)\in \Omega$ and $r>0$ are such that  $P(x_0,t_0,4r) \subseteq \Omega$ and set
\begin{equation}
\bar{w}=\bar{w}^{\pm}_k=\sup_{P(x_0,t_0,4r)}w^{\pm}_k.
\end{equation} 
Writing $W=\bar{w}-w$ we get that, for every non negative $\zeta\in C^{1}_{0}(\Omega)$ 
\begin{equation}
\int -(\partial_tW)\zeta-(a^{ij}\partial_jW)(\partial_i\zeta) \geq \int -c_3|\nabla W|^2\zeta
\end{equation}
so replacing $\zeta$ with $e^{c_4W}\zeta$ for some $c_4=c_4(n)$ we get,using the bounded ellipticity
\begin{equation}
-(\partial_tW)+\partial_i(\bar{a}^{ij}\partial_jW)_i \geq 0
\end{equation} 
for 
\begin{equation}
\bar{a}^{ij}=e^{c_4W}a^{ij}.
\end{equation}
Note that if $\gamma \leq c_5(n)$ a bound on the ellipticity of $\bar{a}^{ij}$ will still be  determined, regardless of $\de$ (as long as $\de<1$). If we are in such a regime, since $W\geq 0$ in $P(x_0,t_0,r)$, by Moser Harnack inequality we get that for some $c_6=c_6(n)$ we have
\begin{equation}\label{mos_har}
r^{-n-2}\int_{ P(x_0,t_0-4r^2,r)}(\bar{w}-w)dx \leq c_6\inf_{P(x_0,t_0,r)}W=c_6\left(\bar{w}-\sup_{P(x_0,t_0,r)}w\right).
\end{equation}
Now, choose $\gamma = 10n\de$ and note that for any sub-domain $\Omega_0 \subseteq \Omega$, choosing $k$  such that $\osc_{\Omega_0}(\partial_ku)\geq \osc_{\Omega_0}(\partial_iu)$ for all $i=1,\ldots,n$ yields 
 \begin{equation} 
8n\de\osc_{\Omega_0}(\partial_ku) \leq \osc_{\Omega_0}(w^{\pm}_k) \leq 12n\de \osc_{\Omega_0}(\partial_ku) 
\end{equation}
which will yield (for $w^{\pm}=w^{\pm}_k$)
\begin{equation}
\begin{aligned}
\inf_{\Omega_0}\left(\bar{w}^+-w^++\bar{w}^--w^-\right)\geq &10n\de(\sup_{\Omega_0} u_k - \inf_{\Omega_0} u_k)+2\inf_{\Omega_0} v-2\sup_{\Omega_0} v \geq \\
&6n\de\osc_{\Omega_0}(u_k) \geq \frac{1}{2}\osc_{\Omega_0}(w^{\pm}).
\end{aligned}
\end{equation}
Thus, there exists some $c_7=c_7(n)$ such that
\begin{equation}
\osc_{P(x_0,t_0,4r)}(w^{\pm}) \leq \frac{c_7}{r^{n+2}} \int_{P(x_0,t_0-4r^2,r)}(\bar{w}^{\pm}-w^{\pm})dx
\end{equation}
holds for \underline{at least one} of $w^{+},w^{-}$. Combining this with \eqref{mos_har}, there exists $c_8=c_8(n)$ such that, assuming the above holds for $w=w^{+}$ and denoting oscillations and sup/inf over the parabolic ball $P(x_0,t_0,\rho)$ by $\osc_\rho,\sup_\rho,\inf_\rho$ we have
\begin{equation}
\osc_{4r}(w) \leq c_8(\sup_{4r}w-\sup_{r}w) \leq c_8(\osc_{4r}w-\osc_{r}(w)). 
\end{equation}

\vspace{5 mm}

\noindent Recapitulating, we see that there exists some $0<\lambda<1$ depending only on $n$ such that if $P(x_0,t_0,4r)\subseteq \Omega$ and $\omega^{\pm}_i(r)=\osc_{P(x_0,t_0,r)}(w^{\pm}_i)$ then:
\begin{enumerate}
\item For every $i\in \{1,\ldots n\}$  we have
\begin{equation}\label{osc_not_inc}
\omega^{\pm}_i(r)\leq \omega^{\pm}_{i}(4r).
\end{equation}
\item If $k$ is chosen such that $\osc_{P(x_0,t_0,4r)}(\partial_ku)\geq \osc_{P(x_0,t_0,4r)}(\partial_iu)$ for $i\in\{1,\ldots,n\}$ then for some sign $\sigma \in \{+,-\}$ we have: 
\begin{equation}\label{osc_dec}
\omega^{\sigma}_k(r)\leq \lambda \omega^{\sigma}_{k}(4r).
\end{equation}
and
\begin{equation}\label{compar_osc}
8n\de\osc_{P(x_0,t_0,4r)}(\partial_ku) \leq \omega^{\sigma}_k(4r) \leq 12n\de \osc_{P(x_0,t_0,4r)}(\partial_ku) 
\end{equation}
\end{enumerate}

\vspace{5 mm}

\noindent After proving the oscillation decay estimate \eqref{osc_dec}, we can easily conclude the proof of the theorem. Let $z_1=(x_1,t_1),z_2=(x_2,t_2)\in \Omega$ and recall that $d_{z_1}=d(z_1,\partial \Omega)$ and $d_{z_1,z_2}=\min(d_{z_1},d_{z_2})$. If $d(z_1,z_2)\geq\frac{1}{4}d_{z_1,z_2}$ the conclusion of the theorem holds trivially, so we may assume that $d(z_1,z_2) < \frac{1}{4}d_{z_1,z_2}$. Assume without loss of generality that $t_1 \geq t_2$, set $R=\frac{1}{4}d_{z_1,z_2}$ and let $m\in \mathbb{N}$ be such that
\begin{equation}
4^{-m-1}R\leq d(z_1,z_2)<4^{-m}R.
\end{equation}
Thus we have a sequence of parabolic balls:
\begin{equation}
z_1,z_2\in P(z_1,4^{-m}R) \subseteq P(z_1,4^{-m+1}R) \subseteq \ldots \subseteq  P(z_1,R) \subseteq \Omega.  
\end{equation}
Setting $(x_0,t_0)=z_1$ so that $\omega^{\pm}_i(r)=\osc_{P(z_0,r)}(w^{\pm}_i)$ we can compare oscillations in $P(z_1,4^{-j+1}R)$ and $P(z_1,4^{-j}R)$ for $j=1,2\ldots, m$ according to \eqref{osc_not_inc},\eqref{osc_dec} and \eqref{compar_osc}. Since at any step one of the $w^{\pm}_i$ decreases by a factor of $\lambda$ and as the are only 2n possibilities for index and sign, we can assume without loss of generality that \eqref{osc_dec},\eqref{compar_osc} happened with $k=1$ and $\sigma=+$ at least $\frac{m}{2n}$ times. This had happened for the last time comparing the oscillations for the parabolic balls  $P(z_1,4^{-j+1}R)$ and $P(z_1,4^{-j}R)$ with $j\geq \frac{m}{2n}$. Thus:
\begin{equation}\label{good_osc}
\begin{aligned}
\osc_{P(z_1,4^{-j+1}R)}(\partial_i u) \leq &\osc_{P(z_1,4^{-j+1}R)}(\partial_1 u) \leq \frac{C}{\de}\omega_1^{+}(4^{-j+1}R) \leq \\ &\frac{C}{\de}\lambda^{m/2n}\omega(R) \leq C\lambda^{m/2n}\de.
\end{aligned}
\end{equation} 
On the other hand $d(z_1,z_2)\geq 4^{-m-1}R$ implies 
\begin{equation}
m \geq \frac{1}{log4}log\left(\frac{R}{4d(z_1,z_2)}\right)
\end{equation}
so with a suitable choice of $0<\alpha<1$ this and \eqref{good_osc} imply:
\begin{equation}
|\nabla u(z_1) - \nabla u(z_2)| \leq C\left(\frac{d(z_1,z_2)}{d_{z_1,z_2}}\right)^{\alpha}\de
\end{equation}
as required $\Box$.

\subsection{Proof of the Main Estimate}\label{conc_sec}

\noindent Before coming back to proving the main estimate, we first derive a crude estimate for the higher derivatives. This estimate ignores the fact that we have control over a thick cylinder and treats it as a union of small parabolic balls.
\begin{lemma}\label{stup_est}
There exists a constant $c$ such that if $\sup_{B(p,r)\times [0,\tau r^2]} |\nabla u| <<1$ then
\begin{equation}
d_z|\nabla^2 u(z)| \leq c
\end{equation}
and
\begin{equation}
d_{z_1,z_2}^{1+\alpha}\frac{|\nabla ^2 u(z_1)-\nabla^2(u(z_2))|}{d(z_1,z_2)^{\alpha}} \leq c
\end{equation}  
\end{lemma}
\proof The first estimate is a direct application of the Ecker Huisken curvature estimate \ref{EH_curv} for balls of radius $d_{z_1}$. For the second part, as before, if $d(z_1,z_2)\geq \frac{1}{4}d_{z_1,z_2}$ there is nothing to prove. Otherwise, using Theorem \ref{EH_curv} once more, we obtain
\begin{equation}
d_{z}^2|\nabla^3 u(z)| \leq c,\;\;\;\;\; d_{z}^3|\partial_t\nabla^2 u(z)| \leq c
\end{equation}
so by integrating, first along space and then along time, we get:
\begin{equation}
|\nabla ^2 u(z_1)-\nabla ^2 u(z_2)|\leq \frac{c}{d_{z_1,z_2}^2}d(z_1,z_2)+\frac{c}{d_{z_1,z_2}^3}d(z_1,z_2)^2 \leq \frac{c}{d_{z_1,z_2}^2}d(z_1,z_2)^{\alpha}d_{z_1,z_2}^{1-\alpha} 
\end{equation}
$\Box$.

\proof[Proof of Theorem \ref{main_est}] Assume again that $r=1$. Let $\tau_0(M,\delta)$, $\gamma_0=\gamma_0(\delta)$ and $\lambda_0=\lambda_0(M,\delta)$  be the constant from Theorem \ref{Schauder}. Fix $\lambda<\lambda_0$, $\tau<\tau_0$ and $\gamma<\gamma_0$ such that the assumptions of the theorem are satisfied. The graphical mean curvature equation is of the form
\begin{equation}
u_t-\Delta u =f
\end{equation}
with 
\begin{equation}
f=\frac{(\partial_iu)(\partial_ju)}{1+|\nabla u|^2}\partial_i\partial_ju.
\end{equation}  
Setting $\alpha$ as in Theorem \ref{holder_grad}, by theorem \ref{Schauder} we have:
\begin{equation}
\sqrt{\tau}|\nabla^2 v(0,\tau)| \leq (1+\delta)\frac{1}{\sqrt{\pi}}\de + \frac{C}{\sqrt{\tau}}\left(\sup_{z_1\in P^{1,\tau}} d^2_{z_1}|f(z_1)|+\sup_{z_1,z_2\in  P^{1,\tau}} d^{2+\alpha}_{z_1,z_2}\frac{|f(z_2)-f(z_1)|}{d(z_2,z_1)^{\alpha}}\right).
\end{equation}
By our assumptions and Lemma \ref{stup_est} 
\begin{equation}
d_z^2|f(z)| \leq Cd_z (d_z |\nabla^2u|)|\nabla u|^2 \leq C\sqrt{\tau}(M\de)^2.
\end{equation}
Similarly, by our assumptions, the H\"older gradient estimate of Theorem \ref{holder_grad} and Lemma \ref{stup_est}:
\begin{equation}
\sup_{z_1,z_2\in  P^{1,\tau}} d^{2+\alpha}_{z_1,z_2}\frac{|f(z_2)-f(z_1)|}{d(z_2,z_1)^{\alpha}}\leq C\sqrt{\tau}(M\de)^2.
\end{equation}
Thus, the contribution of the non-linearity is at most quadratic in the gradient (for small gradients) so for  $\de<\de_0$ it will be smaller than $\delta$. This concludes the proof of the first estimate. The proof of the second estimate is similar, utilizing the second estimate of Theorem \ref{Schauder} $\Box$.

\bibliographystyle{alpha}
\bibliography{ReifBib}

\vspace{10mm}
{\sc Or Hershkovits, Courant Institute of Mathematical Sciences, New York University, 251 Mercer Street, New York, NY 10012, USA}\\

\emph{E-mail:}  or.hershkovits@cims.nyu.edu

\end{document}